\documentclass[12pt]{article}
\usepackage{amssymb}
\usepackage{amsfonts}
\usepackage{amsmath}
\usepackage{array}
\usepackage{enumitem}
\usepackage{tikz}
\usepackage{multicol}
\usepackage{enumitem}
\usepackage{amsthm}
\usepackage{bm}
\usepackage{amssymb}
\usepackage{fullpage} 
\usepackage{url}
\usepackage{verbatim}
\usepackage{mathtools}
\numberwithin{equation}{section}

\RequirePackage[top=1in,bottom=1in,left=1.4in,right=1.4in,includehead]{geometry}

\usepackage[colorlinks,
linkcolor=blue,
anchorcolor=blue,
citecolor=blue
]{hyperref}

\theoremstyle{plain}
\newtheorem{thm}{Theorem}[section]
\newtheorem{lem}[thm]{Lemma}

\newtheorem{cor}[thm]{Corollary}
\newtheorem{conj}[thm]{Conjecture}
\newtheorem{exconj}[thm]{Ex-conjecture}

\theoremstyle{definition}

\def\pod{{\rm{pod}}}

\def\qed{\hfill \rule{4pt}{7pt}}

\newcommand{\qhyp}[5]{
	{_{#1}\phi_{#2}}\bigg(\genfrac{}{}{0pt}{}{#3}{#4};#5\bigg)}
\allowdisplaybreaks[4]

\begin{document}
	\begin{center}
		{\Large \bf {Truncated theta series from the Bailey lattice}}
		\vskip 6mm
		{Xiangyu Ding$^a$,  Lisa Hui Sun$^b$
			\\[2mm]
			Center for Combinatorics, LPMC, Nankai University, Tianjin 300071, P.R. China \\[2mm]
			$^a$dingmath@mail.nankai.edu.cn, $^b$sunhui@nankai.edu.cn \\[2mm]}
	\end{center}
	
	\allowdisplaybreaks
	
	{\noindent \bf Abstract.}  In 2012, Andrews and Merca obtained a truncated version  of Euler's pentagonal number theorem and showed the nonnegativity related to partition functions. Meanwhile, Andrews and Merca, Guo and Zeng  independently conjectured that the truncated Jacobi triple product
	series has nonnegative coefficients, which has been confirmed  analytically and also combinatorially. In 2022, Merca proposed a stronger version for this conjecture.  In this paper, by applying Agarwal, Andrews and Bressoud's identity derived from the Bailey lattice, we obtain a truncated version for the Jacobi triple product series with odd basis, which reduces to  the Andrews--Gordon identity as a special instance. As consequences, we obtain new truncated forms for Euler's pentagonal number theorem, Gauss' theta series on triangular numbers and square numbers,  which lead to inequalities for certain partition functions. Moreover, by considering  a truncated theta series involving $\ell$-regular partitions, we  confirm a conjecture  proposed by Ballantine and Merca about $6$-regular partitions and show that Merca's stronger conjecture on truncated Jacobi triple product series holds when $R=3S$ for $S\geq 1$.
	
	
	
	\allowdisplaybreaks
	
	\section{Introduction}
	In the theory of integer partitions, one of the most important and beautiful identities  is Euler's pentagonal number theorem \cite[p.11]{Partitions},
	\begin{align}\label{Euler}
		(q;q)_\infty=\sum_{j=-\infty}^{\infty}(-1)^j q^{j(3 j+1) / 2},
	\end{align}
	from which  a recurrence relation for the partition function $ p(n) $ can be derived as follows \cite{Partitions}
	\begin{align}
		p(n)+\sum_{j\geq 1}(-1)^j\big(p(n-j(3 j-1) / 2)+p(n-j(3 j+1) / 2)\big)=0,
	\end{align}
	where $p(n)=0$ if $n<0$.
	
	In 2012, Andrews and Merca recaptured Shanks' truncated version \cite[(2)]{shanks1951short} of  Euler's pentagonal number theorem, that is,
	$$\sum_{j=-k}^{k}(-1)^j q^{j(3 j+1) / 2}=\sum_{n=0}^{k} \frac{(-1)^{n}(q ; q)_kq^{\binom{n}{2}+(k+1) n}}{(q ; q)_n},$$
	and obtained the following form \cite[Lemma 1.2]{andrews2012truncated}:
	\begin{align}\label{trun-euler}
		\frac{1}{(q ; q)_{\infty}} \sum_{j=-k}^{k-1}(-1)^j q^{j(3 j+1) / 2}=1+(-1)^{k-1} \sum_{n=1}^{\infty} \frac{q^{\binom{k}{2}+(k+1) n}}{(q ; q)_n}{n-1\brack k-1}.
	\end{align}
In addition, they gave  a combinatorial interpretation for the coefficients of $q^n$ on both sides of   (\ref{trun-euler}) as follows
\begin{align}\label{p(n)-com}
	(-1)^{k-1} \sum_{j=0}^{k-1}(-1)^j \big(p(n-j(3 j+1) / 2)-p(n-j(3 j+5) / 2-1)\big)=M_k(n),
\end{align}
where $M_k(n)$ \cite[Theorem 1.1]{andrews2012truncated} denotes the number of partitions of $n$ in which $k$ is the least integer that is not a part and there are more parts $>k$ than there are $<k$. It directly leads to the following inequality for the partition function $p(n)$
\begin{align}\label{MK--cor-ineq}
	(-1)^{k-1} \sum_{j=-k}^{k-1}(-1)^j\big(p(n-j(3 j+1) / 2\big) \geq 0,
\end{align}
where $n, k \geq 1$ and the strict inequality holds if $n \geq k(3 k+1) / 2$.
In 2015, Kolitsch and  Burnette \cite{kolitsch2015interpreting} reproved  \eqref{p(n)-com}   combinatorially by constructing   special partition pairs.

Later, for Gauss' theta series on triangular and square numbers  \cite[p.23]{Partitions}
\begin{align}
& \frac{\left(q^2 ; q^2\right)_{\infty}}{\left(-q ; q^2\right)_{\infty}}=\sum_{n=0}^{\infty}(-q)^{n(n+1) / 2},\label{Gauss-1} \\
& \frac{(q ; q)_{\infty}}{(-q ; q)_{\infty}}=1+2 \sum_{n=1}^{\infty}(-1)^n q^{n^2},\label{Gauss-2}
\end{align}
Guo and Zeng \cite{GuoZeng13} obtained the following truncated forms, respectively,
\begin{align}
& \frac{\left(-q ; q^2\right)_{\infty}}{\left(q^2 ; q^2\right)_{\infty}} \sum_{j=-k}^{k-1}(-1)^j q^{j(2 j+1)}\notag	\\
& \qquad \qquad = 1+(-1)^{k-1} \sum_{n=k}^{\infty} \frac{\left(-q ; q^2\right)_k\left(-q ; q^2\right)_{n-k} q^{2(k+1) n-k}}{\left(q^2 ; q^2\right)_n}{n-1\brack k-1}_{q^2},\\[8pt]
& \frac{(-q;q)_{\infty}}{(q;q)_{\infty}} \sum_{j=-k}^k(-1)^j q^{j^2}=1+(-1)^k \sum_{n=k+1}^{\infty} \frac{(-q)_k(-1)_{n-k} q^{(k+1) n}}{(q)_n}{n-1\brack k}.
\end{align}
From the above truncated forms, they also showed that    for $n, k \geq 1$, there holds
\begin{align}
&	(-1)^{k-1} \sum_{j=-k}^{k-1}(-1)^j{\pod}\big(n-j(2 j+1)\big) \geq 0,\label{pod-equality}\\
& (-1)^k \sum_{j=-k}^k(-1)^j \overline{p}\left(n-j^2\right) \geq 0,\label{poverline-equality}
\end{align}
where  $\pod(n)$ denotes the number of partitions of $ n  $ with distinct  odd parts and $ \overline{p}(n) $ denotes the number of overpartitions  of $n$ that are  partitions in which the first occurrence of each part may be overlined or not, see, for example  \cite{corteel2004overpartitions,hirschhorn2010arithmetic} for details.

Since then, the truncated forms for theta series, especially for the three classical ones \eqref{Euler}, \eqref{Gauss-1} and \eqref{Gauss-2} have attracted more and more attentions. In 2018, by applying   the Rogers--Fine identity \cite[p.15]{Rogers1917}, Andrews and Merca  \cite{andrews2018truncated} obtained  the truncated forms for  Gauss' theta series \eqref{Gauss-1} and \eqref{Gauss-2} as follows
\begin{align}
& \frac{\left(-q ; q^2\right)_{\infty}}{\left(q^2 ; q^2\right)_{\infty}} \sum_{j=-k}^{k-1}(-1)^j q^{j(2j+1)}\nonumber\\
& \qquad\qquad=1+(-1)^{k-1} \frac{\left(-q ; q^2\right)_k}{\left(q^2 ; q^2\right)_{k-1}} \sum_{j=0}^{\infty} \frac{q^{k(2 j+2 k+1)}\left(-q^{2 k+2 j+3} ; q^2\right)_{\infty}}{\left(q^{2 k+2 j+2} ; q^2\right)_{\infty}}, \label{Andew-18-2}\\[8pt]
&\frac{(-q ; q)_{\infty}}{(q ; q)_{\infty}}\sum_{j=-k}^k(-1)^j q^{j^2}=1+2(-1)^k \frac{(-q ; q)_k}{(q ; q)_k} \sum_{j=0}^{\infty} \frac{q^{(k+1)(k+j+1)}\left(-q^{k+j+2} ; q\right)_{\infty}}{\left(1-q^{k+j+1}\right)\left(q^{k+j+2} ; q\right)_{\infty}}.\label{andrew18-1}	
\end{align}
In 2022, Xia, Yee and Zhao \cite{xia2022new} established more truncated forms for the  classical theta series, such as
\begin{align*}
&\frac{(-q ; q)_{\infty}}{(q ; q)_{\infty}} \sum_{n=0}^{k-1}\left(1-q^{n+1}\right)^2(-1)^n q^{n^2+n} \\[8pt]
&\qquad  =1+(-1)^{k-1} \frac{(-q ; q)_k}{(q ; q)_{k-1}} \sum_{j=0}^{\infty} \frac{q^{k(k+1+j)}\left(-q^{k+2+j} ; q\right)_{\infty}}{\left(q^{k+2+j} ; q\right)_{\infty}}.
\end{align*}
Moreover, Wang and Yee \cite{wang2020truncated} derive many truncated  Hecke--Rogers type  identities.

Noting that the three classical theta series (\ref{Euler}), (\ref{Gauss-1}) and (\ref{Gauss-2}) all fall into the framework  of the Jacobi  triple product series, Andrews and Merca   \cite{andrews2012truncated}, Guo and Zeng  \cite{GuoZeng13}  proposed the following conjecture for truncated Jacobi triple product series,  independently.

\begin{exconj} \label{A-M-conj}
For positive integers $\ell, R, S$ with  $1 \leq S<R / 2$, the coefficient of $q^n$ with $n \geq 1$ in
\begin{align}\label{conj-weak}
(-1)^{\ell-1} \frac{\sum_{j=0}^{\ell-1}(-1)^j q^{R j(j+1) / 2-S j}\left(1-q^{(2 j+1) S}\right)}{\left(q^S, q^{R-S}, q^R ; q^R\right)_{\infty}}
\end{align}
is nonnegative.
\end{exconj}
In 2015, the above conjecture was proved by Mao \cite{mao2015proofs} analytically  and by Yee \cite{yee2015truncated} combinatorially.  He, Ji and Zang \cite{he2016bilateral} considered some truncated series by applying  synchronized $F$-partitions.
Wang and Yee \cite{wang2019truncated} reproved this conjecture by using  Liu's expansion formula for $q$-series \cite{Liuexpan}.

In this paper, we aim to find the truncated theta series for Andrews--Gordon type identities and study the nonnegative properties for the coefficients. Recall that the well-known Andrews--Gordon identity is given by
\begin{align}\label{AGB}
\sum_{n_1, n_2, \ldots, n_{k-1} \geq 0} \frac{q^{N_1^2+N_2^2+\cdots+N_{k-1}^2+N_{i}+\cdots+N_{k-1}}}{(q)_{n_1}(q)_{n_2} \cdots(q)_{n_{k-1}}}=\prod_{\substack{n=1 \\ n \neq 0, \pm i(\bmod 2 k+1)}}^{\infty}\left(1-q^n\right)^{-1},
\end{align}
where $1 \leq i \leq k $, $N_j=n_j+n_{j+1}+\cdots+n_{k-1}$. This identity arises from the first and the second Rogers--Ramanujan identities and has been generalized widely in combinatorics \cite{andrews1975problems, Partitions, bressoud1980analytic, gordon1961combinatorial}, physics \cite{adamovic2009n,adamovic2008n,Berkovich1998,Warnaar1996}, algebra \cite{AFSHARIJOO2023108946,Feigin1993}, representation theory \cite{andrews1999,LEPOWSKY197815} and so on.
From Jennings-Shaffer and Milas' work \cite[Proposition 3.1]{jennings2020q}, one can see that the characters of irreducible modules of the $N = 1$ super-singlet vertex algebra \cite{adamovic2009n,adamovic2008n} can be expressed in terms of the following Andrews--Gordon type identity
\begin{align}\label{J-M-prop}
& \sum_{n_1, n_2, \ldots, n_k \geq 0} \frac{\big(-q^{\frac{1}{2}}\big)_{N_1+\ell}q^{\frac{N_1\left(N_1+1\right)}{2}+\sum_{j=2}^{k}(N_{j}^2+(2\ell+1)N_j)+(\ell+\frac{1}{2})N_{1}
	-\sum_{j=1}^{k-i}N_j}}{(q)_{n_1}(q)_{n_2} \cdots(q)_{n_k}(q)_{n_k+2 \ell}} \notag \\[8pt]
&\qquad\quad = \frac{\big(-q^{\frac{1}{2}}\big)_{\infty}}{(q)_{\infty}}\bigg(\sum_{n \geq|\ell|}-\sum_{n \leq-|\ell|-1}\bigg) q^{\left(k+\frac{1}{2}\right)\left(n+\frac{i}{2 k+1}\right)^2-\frac{(2 k \ell+i+\ell)^2}{4 k+2}},
\end{align}
where $i, k, \ell \in \mathbb{Z}, k \geq 1$ and $0 \leq i \leq k$, $N_j=n_j+n_{j+1}+\cdots+n_{k}$. By noting that not only (\ref{AGB}) but also (\ref{J-M-prop}) can be deduced by the method of Bailey lattice,  we further apply the technique of Bailey lattice to derive the truncated theta series for such identities.

By applying Agarwal, Andrews and Bressoud's Bailey lattice,  we obtain a truncated form of Andrews--Gordon type for the Jacobi triple product series (\ref{conj-weak}) with odd basis.

\begin{thm}\label{lem-thm-1} We have
\[ \begin{aligned}
&\frac{1}{\left(q^{i+1}, q^{2k-i}, q^{2k+1} ; q^{2k+1}\right)_{\infty} }\sum_{j =-\ell} ^{\ell-1}  (-1)^{j} q^{(k+\frac{1}{2}) j^2+(k-i-\frac{1}{2} )j } \\[8pt]
&= 1+\frac{(-1)^{\ell-1}q^{(k+\frac{1}{2})( \ell^{2}+\ell) -(i+1)\ell}(q)_{\infty}}{\left(q^{i+1}, q^{2k-i}, q^{2k+1} ; q^{2k+1}\right)_{\infty}}\notag\\[8pt]
&\times\sum_{n_1, n_2, \ldots, n_k \geq 0} \frac{ (q^{2\ell+1})^{N_1+N_2+\cdots+N_k}
	q^{N_1^{2}+N_2^2+N_3^2+\cdots+N_k^2-N_1-N_2-N_3-\cdots-N_i} (1-q^{\ell})}
{(q)_{n_1}(q)_{n_2} \cdots(q)_{n_{k-1}}
	(q)_{n_{k}}(q)_{n_{k}+2\ell } (1-q^{\ell+n_{k}}) }, \notag
\end{aligned} \]
where  $ i $, $k$, $\ell$ are nonnegative integers with $0 \leq i \leq k$ and $N_j=n_j+n_{j+1}+\cdots+n_{k}$.
\end{thm}

It is notable that when $ \ell =0 $ in the above identity, the sum on the left hand side  vanishes and the terms in the sum on the right hand side are nonzero only when $n_k=0$, thereby it reduces to Andrews--Gordon identity (\ref{AGB}) after replacing $ i$ by $i-1  $.

From the above theorem, when $ k=1 $ and $ i=0 $, the following  truncated  Euler's pentagonal number theorem can be deduced.
\begin{thm}\label{thm-1} For a given positive integer $ \ell $, we have
\begin{align}\label{thm-1-eq}
\frac{1}{(q ; q)_{\infty}} \sum_{j=-\ell}^{\ell-1}(-1)^j q^{j(3 j+1) / 2}=1+(-1)^{\ell-1}\sum_{n=0}^{\infty} \frac{ q^{(2\ell+1)n +n^{2} +\frac{3}{2}\ell ^{2} +\frac{1}{2}\ell}(1-q^{\ell})}{(q;q)_{n} (q;q)_{n+2\ell} (1-q^{n+\ell})}.
\end{align}
\end{thm}
We also show that Theorem \ref{thm-1}  is equivalent to  Andrews and Merca's result (\ref{trun-euler}) .

Further applying the Bailey lattice,  we obtain the truncated forms for Gauss' theta series on triangular numbers and square numbers, respectively.
\begin{thm}\label{thm-2} For a given positive integer $ \ell $, we have
\begin{align}
\frac{\left(-q ; q^2\right)_{\infty}}{\left(q^2 ; q^2\right)_{\infty}} &\sum_{j=-\ell +1}^{\ell}(-1)^j q^{2 j^2+j}\notag \\[8pt]
\qquad&=1+(-1)^{\ell -1}  \sum_{n=0}^{\infty}\frac{ (-q;q^{2})_{n+\ell} q^{2\ell n +n^{2}+2\ell^{2} -\ell}  (1-q^{2\ell})  }{(q^{2};q^{2})_{n} (q^{2};q^{2})_{n+2\ell } (1-q^{2n+2\ell})}.
\end{align}
\end{thm}

\begin{thm}\label{thm-3} For a given positive integer $ \ell $, we have
\begin{align}
\frac{(-q ; q)_{\infty}}{(q ; q)_{\infty}} \sum_{j=-\ell+1 }^{\ell}(-1)^j q^{j^2}= 1+(-1)^{\ell-1} \sum_{n=0}^{\infty } \frac{(-q)_{n+\ell}  q^{\ell n+\frac{n(n+1)}{2}+\ell^{2}}  (1-q^{\ell}) }{(q)_{n} (q)_{n+2\ell} (1-q^{n+\ell})}.
\end{align}
\end{thm}

Recently, Merca \cite{merca2022two} proposed a  stronger version of Ex-conjecture \ref{A-M-conj} as follows.
\begin{conj}\label{stronger-ja}
For positive integers $1\leq S< R$ with $k \geq 1$, the theta series
\begin{align}\label{conj-strong}
(-1)^{k} \frac{\sum_{j=k}^{\infty}(-1)^j q^{R j(j+1) / 2}\left(q^{-Sj}-q^{( j+1) S}\right)}{\left(q^S, q^{R-S}; q^R\right)_{\infty}}
\end{align}
has nonnegative coefficients.
\end{conj}
Very recently, Ballantine and Feigon \cite{ballantine2024truncated} proved the cases for $k\in\{1, 2, 3\} $ of the above conjecture. They also derived a result \cite[Theorem 1.2]{ballantine2024truncated}  combinatorially which is weaker than the special case  $S=1  $ and $ R=3 $. In Section \ref{BM-con}, we shall provide an analytic proof for the cases when $ R=3S $ for $S\geq 1$, which is equivalent to the case $S=1$ and $R=3$ by replacing $q^S$ by $q$.

\begin{thm}\label{R=3-stronger}   Conjecture \ref{stronger-ja}  holds when $ R=3S $ for positive integers $S\geq1$.
\end{thm}

As another consequence given in Section \ref{BM-con}, we confirm the  following  conjecture related to $6$-regular partitions proposed by Ballantine and Merca \cite{ballantine20236}
and further certify that it is still true for the more general cases with $\ell\geq 3$. Recently, Yao \cite{Yao6regular} proved it by using a formula of the number of partitions of $n$ into parts not exceeding 3 due to Cayley.

\begin{conj}[Ballantine and Merca {\cite[Conjecture 2]{ballantine20236}}] \label{B-M-conj-1}
For $n \geq 0$, $k\geq 1$,
$$
(-1)^{k}\Big(\alpha_n-\sum_{j=-k+1}^{k}(-1)^j b_6(n-j(3 j-1) / 2)\Big) \geq 0,
$$
with strict inequality if $n \geq k(3 k+1) / 2$, where $ b_6(n) $ is the number of the 6-regular partitions of n, and
$$
\alpha_n:= \begin{cases}(-1)^m, & \text { if } n=3 m(3 m-1), m \in \mathbb{Z} \\ 0, & \text { otherwise. }\end{cases}
$$
\end{conj}

This paper is organized as follows. In Section \ref{pre}, we shall introduce some basic definitions and notations for $q$-series, Bailey pairs, Bailey lattices and integer partitions.   In Section  \ref{pr-trunja}, by using Agarwal, Andrews and Bressoud's Bailey lattice,  we give the truncated forms for the Jacobi triple product theta series with odd and even basis, respectively. In Section \ref{pr-Gauss-2}, we will deduce the truncated versions for the  three classical theta series and their corresponding partition function inequalities. In Section \ref{BM-con}, we prove and extend Ballantine and Merca's conjectures to $\ell$-regular partitions with $\ell=3$ or $\ell\geq 6$. We also show that the stronger conjecture on the truncated Jacobi triple product series holds when $R=3S$ for $S\geq 1$. In Section \ref{conclusion}, we remark the relations of truncated theta series with other Bailey chain and Bailey lattice.

\section{Preliminaries}\label{pre}

Throughout this paper, we adopt  standard notations and terminologies
for $q$-series, see,  for example, \cite{Basic}. We assume that $|q|<1$.
The $q$-shifted factorial is denoted by
\[
(a)_n=(a;q)_n=\begin{cases}
1, & \text{\it if $n=0$}, \\[2mm]
(1-a)(1-aq)\cdots(1-aq^{n-1}), & \text{\it if $n\geq 1$}.
\end{cases}
\]
We also use the notation
\[
(a)_\infty=(a;q)_\infty=\prod_{n=0}^\infty (1-aq^n).
\]
There are more compact notations for the multiple $q$-shifted factorials
\begin{align*}
(a_1,a_2,\dots,a_m;q)_n&=(a_1;q)_n(a_2;q)_n \cdots(a_m;q)_n,\\
(a_1,a_2,\dots,a_m;q)_{\infty}&=(a_1;q)_{\infty}(a_2;q)_{\infty}\cdots
(a_m;q)_{\infty}.
\end{align*}
The Gaussian polynomials, or $q$-binomial coefficients are given by
\[
{n \brack k} =
\left\{
\begin{array}{ll}
0, & \hbox{if $k<0$ or $k>n$},\\
\dfrac{(q;q)_n}
{(q;q)_k(q;q)_{n-k}}, & \hbox{otherwise.}
\end{array}
\right.
\]
We also denote the case when $q\mapsto q^t$ by ${n\brack k}_{q^{t}}$ with $t\geq 2$.

Recall that the
basic hypergeometric series $_r\phi_s$ is defined as follows
\[
\qhyp{r}{s}{a_1,a_2,\dots,a_r}{b_1,b_2,\dots,b_s}{q,x}
=\sum_{n=0}^{\infty}\frac{(a_1,a_2,\dots,a_r;q)_n}
{(q,b_1,\dots,b_s;q)_n} \Big[(-1)^n q^{\binom{n}{2}}\Big]^{1+s-r} x^n.
\]
To prove the results in this paper, we will need the following transformation
formulas for basic hypergeometric series.
The well-known transformation formula   for $_2\phi_1$ series due to
Heine is \cite[(III.3)]{Basic}
\begin{align}\label{Heine-2}
{ }_2 \phi_1(a, b ; c ; q, z)=\frac{(a b z / c ; q)_{\infty}}{(z ; q)_{\infty}}{ }_2 \phi_1(c / a, c / b ; c ; q, a b z / c),
\end{align}
where $  \text{max}\{|z|,|abz/c|\}<1$ so that both $  _{2}\phi_{1} $ series converge.
Two of the transformation formulas for ${_3\phi_2}$ series are \cite[(III.9)-(III.10)]{Basic}
\begin{align}
&\qhyp{3}{2}{a,b,c}{d,e}{q,\frac{de}{abc}}\notag\\
&=\frac{(e/a,de/bc;q)_\infty}{(e,de/abc;q)_\infty}\,
\qhyp{3}{2}{a,d/b,d/c}{d,de/bc}{q,\frac{e}{a}}\label{3phi2-1}\\
&=\frac{(b,de/ab,de/bc;q)_\infty}{(d,e,de/abc;q)_\infty}\,
\qhyp{3}{2}{d/b,e/b,de/abc}{de/ab,de/bc}{q,b} \label{3phi2-2}
\end{align}
provided that $\max\{|b|, |e/a|, |de/abc|\}<1$.
The Rogers--Fine identity \cite[p.15]{Rogers1917} is given by
\begin{align}\label{R-R-id}
\sum_{n=0}^{\infty} \frac{(\alpha ; q)_n \tau^n}{(\beta ; q)_n}=\sum_{n=0}^{\infty} \frac{(\alpha ; q)_n(\alpha \tau q / \beta ; q)_n \beta^n \tau^n q^{n^2-n}\left(1-\alpha \tau q^{2 n}\right)}{(\beta ; q)_n(\tau ; q)_{n+1}},
\end{align}
where $|\tau|<1$.
The well-known Jacobi triple product identity is given by \cite[II.28]{Basic}
\begin{equation}\label{JTPI}
\sum_{k=-\infty}^\infty (-1)^k q^{k\choose 2} a^k =(a,q/a,q;q)_\infty.
\end{equation}

Bailey in \cite{Bailey1949} discovered a powerful tool to prove and construct  $q$-identities, which is given as a pair of sequences of rational functions $(\alpha_n, \beta_n)_{n\geq 0}$  with respect to $a$ such that
\begin{equation}\label{Baileyp}
\beta_n=\sum_{j=0}^n \frac{\alpha_j}{(q;q)_{n-j}(aq;q)_{n+j}}.
\end{equation}
It was named Bailey pair and has been studied widely. In \cite{agarwal1987bailey}, Agarwal, Andrews and Bressoud gave the following identity deduced from the Bailey lattice.
\begin{lem}\label{bailey-lattice}
Let $\alpha_n, \beta_n$ be sequences satisfying equation  (\ref{Baileyp}) and let $0 \leq i \leq k$, then
\[ 		\begin{aligned}
& \sum_{n \geq m_1 \geq m_2 \geq \ldots>m_k \geq 0} \frac{\left(\rho_1\right)_{m_1} \ldots\left(\rho_k\right)_{m_k}\left(\sigma_1\right)_{m_1} \ldots\left(\sigma_k\right)_{m _{k}}}{(q)_{n-m_1} \ldots(q)_{m _{k-1}-m _{k}}}  \\[8pt]
&\times\frac{\left(a / \rho_1 \sigma_1\right)_{n-m_1} \cdots\left(a / \rho_i \sigma_i\right)_{m_{i-1}-m_i}}{\left(a / \rho_1\right)_n \cdots\left(a / \rho_i\right)_{m_{i-1}}\left(a / \sigma_1\right)_n \cdots\left(a / \sigma_i\right)_{m_{i-1}}}  \\[8pt]
& \times \frac{\left(a q / \rho_{i+1} \sigma_{i+1}\right)_{m_i-m_{i+1}} \cdots\left(a q / \rho_k \sigma_k\right)_{m_{k-1}-m_k} a^{m_1+\cdots+m_k} q^{m_{i+1}+\cdots+m_k} \beta_{m _{k}}}{\left(a q / \rho_{i+1}\right)_{m_i} \cdots\left(a q / \rho_k\right)_{m _{k-1}}\left(a q / \sigma_{i+1}\right)_{m_i} \cdots\left(a q / \sigma_k\right)_{m _{k-1}}\left(\rho_1 \sigma_1\right)^{m_1} \cdots\left(\rho_k \sigma_k\right)^{m_k}} \\[8pt]
=&\ \frac{\alpha_0}{(q)_n(a)_n}+\sum_{t=1}^n \frac{\left(\rho_1,\sigma_1, \ldots,\rho_i,\sigma_i \right)_t a^{it}(1-a)}{\left(a / \rho_1,a / \sigma_1, \ldots,a / \rho_i,a / \sigma_i \right)_t\left(\rho_1 \cdots \rho_i \sigma_1 \cdots \sigma_i\right)^t(q)_{n-t}(a)_{n+t}}  \\[8pt]
&	\times\bigg(\frac{\left(\rho_{i+1},\sigma_{i+1}, \ldots,\rho_k,\sigma_k \right)_t(a q)^{(k-i) t} \alpha_t}{\left(a q / \rho_{i+1},a q / \sigma_{i+1},\ldots, a q / \rho_k,a q / \sigma_k \right)_t\left(\rho_{i+1} \cdots \rho_k \sigma_{i+1} \cdots \sigma_k\right)^t\left(1-a q^{2 t}\right)} \\[8pt]
&-\frac{\left(\rho_{i+1},\sigma_{i+1},\ldots,\rho_k,\sigma_k \right)_{t-1}(a q)^{(k-i)(t-1)} a q^{2 t-2} \alpha_{t-1}}{(a q / \rho_{i+1}, a q / \sigma_{i+1}, \ldots, a q / \rho_k,a q / \sigma_k )_{t-1}(\rho_{i+1} \cdots \rho_k \sigma_{i+1} \cdots \sigma_k)^{t-1}(1-a q^{2 t-2})}\bigg).
\end{aligned} \]
\end{lem}

Based on the above result, we mainly adopt the following Bailey pair which is deduced from a $_{6}\phi_{5} $ series \cite[(II.21)]{Basic} to derive our main results
\begin{align}\label{BP6phi5}
&\alpha_n=\frac{\left(1-a q^{2 n}\right)}{(1-a)} \frac{(-1)^{n} q^{\binom{n}{2}} (a)_{n} (b)_{n} (c)_{n}}{(q)_{n} (aq/b)_{n} (aq/c)_{n}} \left(\frac{aq}{bc}\right)^{n},  \nonumber \\[8pt]
&\beta_n=\frac{(a q / b c ; q)_n}{(q,a q / b, a q /c ; q)_n}.
\end{align}
Furthermore, letting $ b\mapsto q $, $ c\mapsto (aq)^{\frac{1}{2}} $ in the above Bailey pair, we obtain
\begin{align}\label{Bailey-pair-1}
\alpha_n=\frac{\left(1-a q^{2 n}\right)}{(1-a)} (-1)^{n} q^{\binom{n}{2}}(a^{\frac{1}{2}}q^{-\frac{1}{2}})^{n}, \quad
\beta_n=\frac{(a^{\frac{1}{2}}q^{-\frac{1}{2}}; q)_n}{(q,a , a^{\frac{1}{2}}q^{\frac{1}{2}} ; q)_n}.
\end{align}
Letting $ b\mapsto q $, $ c\to \infty$ in \eqref{BP6phi5}, we also have  another Bailey pair
\begin{align}\label{Bailey-pair-2}
\alpha_n=\frac{\left(1-a q^{2 n}\right)}{(1-a)} a^{n} q^{n^{2}-n}, \quad
\beta_n=\frac{1}{{(q,a} ; q)_n}.
\end{align}

Based on the truncated theta series, we will further study the inequalities  related to certain integer partitions. Recall that a partition of a nonnegative integer $n$ is a finite nonincreasing sequence of positive integers $\lambda_1,\lambda_2,\ldots, \lambda_k$ such that $\sum_{i=1}^k \lambda_i=n$.
Let $p(n)$ denote  the number of integer partitions of $n$ with $p(0)=1$. It is known that the generating function for $p(n)$ is
\[
\sum_{n=0}^\infty p(n)q^n=\frac{1}{(q;q)_\infty},
\]
see, for example \cite{Partitions}.
Hirschhorn and Sellers \cite{hirschhorn2010arithmetic} introduced the notion  $\pod(n)$ which denotes the number of partitions of $ n $ with distinct  odd parts that has the following generating function
\begin{equation}\label{gf-pod}
\sum_{n=0}^{\infty} {\pod}(n)q^{n}=\frac{(-q;q^{2})_{\infty}}{(q^{2};q^{2})_{\infty}}.
\end{equation}
Corteel and Lovejoy  \cite{corteel2004overpartitions} introduced overpartitions which are partitions of $ n $ whose first occurrence of a part may be overlined or not. Let $ \overline{p}(n) $ denote  the number of overpartitions then its generating function is
\begin{equation}\label{gf-over}
\sum_{n=0}^{\infty} \overline{p}(n)q^{n}=\frac{(-q;q)_{\infty}}{(q;q)_{\infty}}.
\end{equation}
The $\ell$-regular partitions are the partitions with no part  divisible by $ \ell $, which  parameterize the irreducible $ \ell $-modular representations of the symmetric group when $ \ell  $ is prime  \cite{RepJames}. Let $ b_{\ell}(n) $ denote  the number of $\ell$-regular partitions of $n$, then its generating function is given by
\begin{equation}\label{gf-regular}
\sum_{n=0}^{\infty} b_{\ell}(n)q^{n}=\frac{(q^{\ell};q^{\ell})_{\infty}}{(q;q)_{\infty}}.
\end{equation}
We also denote $p_{5}(n)$  the number of partitions of $n$ only with parts $1$, $2$, $4$ or $5$, whose generating function is given by
\begin{equation}\label{gf-p5}
\sum_{n=0}^{\infty} p_{5}(n)q^{n}=\frac{1}{(1-q)(1-q^2)(1-q^4)(1-q^{5})},
\end{equation}
see \cite[A029007]{OEIS}, for example.

\section{The Andrews--Gordon type truncated Jacobi  triple product series}\label{pr-trunja}

By employing  Agarwal, Andrews and Bressoud's identity as given in Lemma \ref{bailey-lattice}, we can obtain the Andrews--Gordon type truncated form for the Jacobi triple product series with odd and even basis.

Firstly, note that by  letting $ n\to \infty  $, $ \rho_{s} $, $\sigma_{s}$ $\to \infty$ for $ s=2,3,\ldots, k$, and $ \rho_{1}\mapsto\rho$, $ \sigma_{1} \mapsto\sigma$ in Lemma  \ref{bailey-lattice}, we can simplify the identity to the following form
\begin{align}\label{even-pf}
& \sum_{m_1 \geq m_2 \geq \cdots \geq m_k \geq 0} \frac{(\rho,\sigma;q)_{m_1} a^{m_1+m_2+\cdots+m_k} q^{m_2^2+m_3^2+\cdots+m_k^2-m_2-m_3-\cdots-m_i} \beta_{m_k}}{  (\rho \sigma)^{m_{1}}   (q)_{m_1-m_2}(q)_{m_2-m_3} \cdots(q)_{m_{k-1}-m_k}} \notag \\[8pt]
&=  \frac{(a/\rho,a/\sigma;q)_{\infty}}{(a/ \rho\sigma,a q;q)_{\infty}} \sum_{n \geq 0} \frac{ (\rho,\sigma;q)_{n} a^{k n} q^{\left(k-1\right) n^2+\left(1-i\right) n} \alpha_n}{  (\rho \sigma)^{n}  (a/\rho,a/\sigma;q)_{n}\left(1-a q^{2 n}\right)} \notag\\[8pt]
&\quad- \frac{(a/\rho,a/\sigma;q)_{\infty}}{(a/ \rho\sigma,a q;q)_{\infty}} \ \sum_{n \geq 1} \frac{ (\rho,\sigma;q)_{n} a^{k n+i-k+1} q^{\left(k-1\right) n^2+\left(3+i-2 k\right) n+k-i-2} \alpha_{n-1}}{(\rho \sigma)^{n} (a/\rho,a/\sigma;q)_{n}\left(1-a q^{2 n-2}\right)}.
\end{align}	
By  inserting into different Bailey pairs $(\alpha_n, \beta_n)$, it will lead  to different theta series identities.   Let us first consider the odd basis case and give the proof of Theorem  \ref{lem-thm-1}.

{\noindent \it  Proof of Theorem \ref{lem-thm-1}}. Setting $ \sigma\to \infty $ in (\ref{even-pf}), it leads to that
\begin{align}\label{pf-12}
& \sum_{m_1 \geq m_2 \geq \cdots \geq m_k \geq 0} \frac{(\rho)_{m_1}(-\rho )^{-m_1}  a^{m_1+m_2+\cdots+m_k} q^{\frac{m_1\left(m_1-1\right)}{2}+m_2^2+m_3^2+\cdots+m_k^2-m_2-m_3-\cdots-m_i} \beta_{m_k}}{(q)_{m_1-m_2}(q)_{m_2-m_3} \cdots(q)_{m_{k-1}-m_k}} \notag \\[8pt]
&\quad =  \frac{(a/\rho)_{\infty}}{(a q)_{\infty}} \sum_{n \geq 0} \frac{(-\rho )^{-n} (\rho)_{n} a^{k n} q^{\left(k-\frac{1}{2}\right) n^2+\left(\frac{1}{2}-i\right) n} \alpha_n}{(a/\rho)_{n}\left(1-a q^{2 n}\right)} \notag\\[8pt]
&\qquad -\frac{(a / \rho)_{\infty}}{(a q)_{\infty}} \sum_{n \geq 1} \frac{(-\rho )^{-n} (\rho)_{n} a^{k n+i-k+1} q^{\left(k-\frac{1}{2}\right) n^2+\left(\frac{5}{2}+i-2 k\right) n+k-i-2} \alpha_{n-1}}{(a/\rho)_{n}\left(1-a q^{2 n-2}\right)}.
\end{align}
Taking the limit $ \rho\to \infty  $, then inserting the Bailey pair (\ref{Bailey-pair-1}) into (\ref{pf-12}) and letting $ a\mapsto q^{2\ell +1}$, we  have
\begin{align}\label{progress-pr-1}
& \sum_{m_1 \geq m_2 \geq \cdots \geq m_k \geq 0} \frac{  (q^{2\ell+1})^{{m_1}+m_2+\cdots+m_k} q^{{m_1^{2}+m_2^2+m_3^2+\cdots+m_k^2-m_1-m_3-\cdots-m_i}} (1-q^{\ell })}{(q)_{m_1-m_2}(q)_{m_2-m_3} \cdots(q)_{m_{k-1}-m_k} (q)_{m_{k}} (q^{2\ell+1}) _{m_{k}} (1-q^{m_{k} +\ell})} \notag \\
&=  \frac{1}{( q^{2\ell+1})_{\infty}} \sum_{n \geq 0}(-1)^{n} { (q^{2\ell+1})^{(k+\frac{1}{2}) n} q^{(k+\frac{1}{2}) n^2-(i+1 )n } } \notag\\
& \quad -\frac{{1}}{(q^{2\ell+1})_{\infty}} \sum_{n \geq 1} (-1)^{n-1} (q^{2\ell+1})^{(k+\frac{1}{2})n+i-k+\frac{1}{2}} q^{(k+\frac{1}{2}) n^2+(i-2 k) n+k-i-\frac{1}{2}} \notag\\
&=  \frac{1}{( q^{2\ell+1})_{\infty}} \sum_{n \geq 0}(-1)^{n} {  q^{(k+\frac{1}{2}) (n+\ell)^2+(k-i-\frac{1}{2} )n -(k+\frac{1}{2})\ell^{2}} } \notag\\
&\quad +\frac{{1}}{(q^{2\ell+1})_{\infty}} \sum_{n \geq 1} (-1)^{n}  q^{(k+\frac{1}{2}) (n+\ell)^2+(i- k+\frac{1}{2}) n+2\ell(i-k+\frac{1}{2}) -(k+\frac{1}{2})\ell^{2} }.
\end{align}
Substituting $ n\mapsto n-\ell $ in the first sum and $ n\mapsto -n-\ell $ in the second sum in the above result, it implies that
\begin{align}\label{thm-1-pf}
& \sum_{m_1 \geq m_2 \geq \cdots \geq m_k \geq 0} \frac{  (q^{2\ell+1})^{{m_1}+m_2+\cdots+m_k} q^{{m_1^{2}+m_2^2+m_3^2+\cdots+m_k^2-m_{1}-m_2-m_3-\cdots-m_i}} (1-q^{\ell })}{(q)_{m_1-m_2}(q)_{m_2-m_3} \cdots(q)_{m_{k-1}-m_k} (q)_{m_{k}} (q^{2\ell+1}) _{m_{k}} (1-q^{m_{k} +\ell})} \notag\\[8pt]
&=  (-1)^{\ell}\frac{q^{-(k+\frac{1}{2})\ell ^{2} -(k-i-\frac{1}{2})\ell}}{( q^{2\ell+1})_{\infty}} \notag\\[8pt]
&\quad \times \bigg( \sum_{n \geq \ell}(-1)^{n}   q^{(k+\frac{1}{2}) n^2+(k-i-\frac{1}{2} )n }
+\sum_{n \leq-\ell- 1} (-1)^{n}  q^{(k+\frac{1}{2}) n^2+( k-i-\frac{1}{2}) n} \bigg).
\end{align}
Multiplying both sides of the above equation by $ 1/(q)_{2\ell} $, reindexing with $m_{j}=N_j=n_j+n_{j+1}+\cdots+n_{k}$ and applying the Jacobi  triple product identity \eqref{JTPI}, we complete the proof. \qed

The even basis case for the truncated Jacobi triple product series can be derived similarly as follows.
\begin{thm}\label{thm-even} We have
\[ \begin{aligned}
&\frac{1}{\left(q^{k-i}, q^{k+i}, q^{2k} ; q^{2k}\right)_{\infty} }\sum_{n =-\ell+1} ^{\ell-1}  (-1)^{n} q^{kn^2-in } \\[8pt]
&= 1+ \frac{2(-1)^{\ell-1}q^{k\ell^{2}-i\ell}}{\left(q^{k-i}, q^{k+i}, q^{2k} ; q^{2k}\right)_{\infty} }\\
&\quad \times \sum_{n_1, n_2, \ldots, n_k \geq 0} \frac{(-1)^{N_{1}} (q,-q)_{N_1+\ell-1}q^{2\ell( N_2+\cdots+N_k) +N_2^2+N_3^2+\cdots+N_k^2-N_2-N_3-\cdots-N_i} }{(q)_{n_1}(q)_{n_2} \cdots(q)_{n_{k-1}}(q)_{n_{k}} (q )_{n_{k}+2\ell-1}},
\end{aligned} \]
where  $ i $, $ k $,  $\ell$ are positive integers with $ 1\leq i \leq k-1 $    and $N_j=n_j+n_{j+1}+\cdots+n_{k}$.
\end{thm}

\begin{proof}
Inserting the Bailey pair (\ref{Bailey-pair-2}) into (\ref{even-pf}),  substituting  $ \rho \mapsto  a^{\frac{1}{2}} $, $ \sigma \mapsto -a^{\frac{1}{2}}$ and then setting $a\mapsto q^{2\ell} $,   we  have
\begin{align}
& \sum_{m_1 \geq m_2 \geq \cdots \geq m_k \geq 0} \frac{(-1)^{m_{1}}(q^{\ell},-q^{\ell};q)_{m_1} (q^{2\ell})^{m_2+\cdots+m_k} q^{m_2^2+m_3^2+\cdots+m_k^2-m_2-m_3-\cdots-m_i} }{(q)_{m_1-m_2}(q)_{m_2-m_3} \cdots(q)_{m_{k-1}-m_k}(q,q^{2\ell})_{m_{k}}} \nonumber \\[8pt]
&=\frac{(q^{\ell},-q^{\ell};q)_{\infty}}{(-1,q^{2\ell};q)_{\infty}} \Big(\sum_{n \geq 0} (-1)^{n} q^{k(n^{2}+2\ell n)-in}-   \sum_{n \geq 1} (-1)^{n} q^{k(n^{2}+2(\ell-1)n )+in+(2\ell-1)(i-k)}\Big)\nonumber \\
&=\frac{(q^{\ell},-q^{\ell};q)_{\infty}}{(-1,q^{2\ell};q)_{\infty}}(-1)^{\ell}q^{-k\ell^{2}+i\ell} \Big(\sum_{n \geq \ell} (-1)^{n} q^{kn^{2}-in}+   \sum_{n \leq -\ell} (-1)^{n} q^{kn^{2}-in}\Big), \label{pfevenmod}
\end{align}
where the last step follows from substituting $n\mapsto n-\ell$ in the first sum and $n \mapsto -n-\ell+1$ in the second sum.
It completes the proof by further applying the Jacobi triple product identity \eqref{JTPI} and  replacing $m_j$  with $N_j=n_j+n_{j+1}+\cdots+n_{k}$.
\end{proof}

Note that if we set $ \ell=0 $ in the procedure to prove Theorem \ref{thm-even}, it leads to the following result.

\begin{cor} We have
\begin{multline}\label{bressoudeven}
\sum_{m_1 \geq m_2 \geq \cdots \geq m_k \geq 1} \frac{(-1)^{m_{1}} (q,-q;q)_{m_1-1}  q^{m_2^2+m_3^2+\cdots+m_k^2-m_2-m_3-\cdots-m_i} }{(q)_{m_1-m_2}(q)_{m_2-m_3} \cdots(q)_{m_{k-1}-m_k}(q)_{m_{k}}(q)_{m_{k}-1}}\\[7pt]=\frac{1}{2}\left(q^{k-i}, q^{k+i}, q^{2k} ; q^{2k}\right)_{\infty},
\end{multline}
where  $ i $, $ k $ are positive integers with $ 1\leq i \leq k $.
\end{cor}

\begin{proof}
Taking the limit $\ell\to 0$ in \eqref{pfevenmod}, it reduces to that
\[ \begin{aligned}
& \sum_{m_1 \geq m_2 \geq \cdots \geq m_k \geq 0} \frac{(-1)^{m_{1}} (-1;q)_{m_{1}}(q  ;q)_{m_1-1}  q^{m_2^2+m_3^2+\cdots+m_k^2-m_2-m_3-\cdots-m_i} }{(q)_{m_1-m_2}(q)_{m_2-m_3} \cdots(q)_{m_{k-1}-m_k}(q)_{m_{k}} (q )_{m_{k}-1}} \lim_{\ell \to 0}\Big(\frac{1-q^{\ell}}{1-q^{2\ell}}\Big) \notag \\[8pt]
&\qquad=\lim_{\ell \to 0}\Big(\frac{1-q^{\ell}}{1-q^{2\ell}}\Big)\left(1+\sum_{n = -\infty}^{\infty} (-1)^{n} q^{kn^{2}-in}  \right).
\end{aligned} \]
Since when $ m_{k}=0 $, the multiple sum on the left hand side vanishes unless $  m_{1}=\cdots=m_{k-1}=0 $, we are led to that
\[ \begin{aligned}
&\quad\  1+\sum_{m_1 \geq m_2 \geq \cdots \geq m_k \geq 1} \frac{(-1;q)_{m_{1}}(q  ;q)_{m_1-1}(-1)^{m_{1}}   q^{m_2^2+m_3^2+\cdots+m_k^2-m_2-m_3-\cdots-m_i} }{(q)_{m_1-m_2}(q)_{m_2-m_3} \cdots(q)_{m_{k-1}-m_k}(q)_{m_{k}} (q )_{m_{k}-1}} \notag \\[8pt]
&\qquad =1+\sum_{n = -\infty}^{\infty} (-1)^{n} q^{kn^{2}-in} .
\end{aligned} \]
Thus the proof can be  completed by further employing the Jacobi triple product identity \eqref{JTPI} and simplifying.
\end{proof}

Note that by substituting $i$ with $k-i$ in \eqref{bressoudeven}, it could be related to
Bressoud's identity with even basis \cite{bressoud1980analytic}, which is given as follows
\[\begin{aligned}
\sum_{m_1  \geq \cdots \geq m_{k-1} \geq 0} \frac{q^{m_1^2+m_2^2+\cdots+m_{k-1}^2+m_{i}+\cdots+m_{k-1}}}{(q)_{m_1-m_2}  \cdots(q)_{m_{k-2}-m_{k-1}}(q^{2};q^{2})_{m_{k-1}}}= \frac{\left(q^{i}, q^{2k-i}, q^{2k} ; q^{2k}\right)_{\infty}}{(q;q)_{\infty}}.
\end{aligned}\]

\section{The three classical  truncated theta series}\label{pr-Gauss-2}

In this section, we will study the truncated forms for Euler's pentagonal number theorem, Gauss'  theta series on triangular numbers and square numbers. As applications, the nonnegativity of the corresponding coefficients   can be directly derived  which  will lead to some  inequalities for certain partition functions.

It is straightforward to see that the truncated Euler's pentagonal number theorem as given in Theorem \ref{thm-1}  can be deduced by taking   $ k=1 $ and  $ i=0 $ in Theorem \ref{lem-thm-1}. First, we show  that our truncated form for Euler's pentagonal number theorem is equivalent to Andrews and Merca's result \eqref{trun-euler} by verifying the following identity.

\begin{thm}\label{equi-1}  For $\ell$ a  positive integer, we have
\[ \begin{aligned}
\sum_{n=\ell}^{\infty} \frac{q^{\binom{\ell}{2}+(\ell+1) n}}{(q ; q)_n}{n-1\brack \ell-1}=q^{\frac{3}{2}\ell ^{2}+\frac{1}{2}\ell} \sum_{n=0}^{\infty} \frac{ q^{(2\ell+1)n +n^{2}} (1-q^{\ell})}{(q;q)_{n} (q;q)_{n+2\ell} (1-q^{n+\ell})}.
\end{aligned} \]
\end{thm}
\begin{proof} We see that the right hand side can be rewritten as follows
\begin{align}
&\quad \ q^{\frac{3}{2}\ell ^{2}+\frac{1}{2}\ell} \sum_{n=0}^{\infty} \frac{ q^{(2\ell+1)n +n^{2}} (1-q^{\ell})}{(q;q)_{n} (q;q)_{n+2\ell} (1-q^{n+\ell})}\nonumber\\[8pt]
&= \frac{q^{\frac{3}{2}\ell ^{2}+\frac{1}{2}\ell} }{(q;q)_{2\ell}}\sum_{n=0}^{\infty} \frac{ q^{(2\ell+1)n +n^{2}} (q^{\ell};q)_{n}}{(q;q)_{n} (q^{2\ell+1};q)_{n} (q^{\ell+1};q)_{n}}\nonumber\\[8pt]
&=\frac{q^{\frac{3}{2}\ell ^{2}+\frac{1}{2}\ell} }{(q;q)_{2\ell}} \frac{ 1}{(q^{2\ell+1};q)_{\infty}}\sum_{n=0}^{\infty} \frac{ (-1)^{n} q^{(2\ell+1)n +\frac{n^{2} -n}{2}} (q;q)_{n}}{(q;q)_{n}  (q^{\ell+1};q)_{n}}\nonumber \\[8pt]
&\quad(\text{by}\ (\ref{3phi2-1}) \text{ with } a,\ c\to \infty,\ b= q^{\ell},\ d= q^{\ell+1},\ e=q^{2\ell+1})\nonumber\\[8pt]
&=\frac{q^{\frac{3}{2}\ell ^{2}+\frac{1}{2}\ell} }{(q;q)_{\infty}}
(q^{\ell+1};q)_{\infty} \sum_{n=0}^{\infty} \frac{ q^{(\ell+1)n } (q^{\ell};q)_{n}}{(q;q)_{n}  (q^{\ell+1};q)_{n}}\notag \\[8pt]
&\quad(\text{ by}\  (\ref{Heine-2})\  \text{with} \ b\to 0,\ a= q^{\ell},\ z=c=q^{\ell+1})\nonumber\\[8pt]
&=\sum_{n=\ell}^{\infty} \frac{q^{\binom{\ell}{2}+(\ell+1) n}}{(q ; q)_n}{n-1\brack \ell-1}
\qquad(\text{  by shifting $ n $ to $ n-\ell $}).\nonumber
\end{align}
\end{proof}

Furthermore, from our result  \eqref{thm-1-eq} and by using the following expansion
\begin{equation}\label{formulag}
\frac{1-q^\ell}{1-q}=\sum_{i=0}^{\ell-1} q^i,
\end{equation}
it is direct to see that the inequality for partition functions (\ref{MK--cor-ineq}) holds.
More precisely, from \eqref{thm-1-eq} we derive that
\[ \begin{aligned}
&\frac{1}{(q; q)_{\infty}} \sum_{n=-\ell}^{\ell-1}(-1)^n q^{n(3 n+1) / 2}\\
&	=1+(-1)^{\ell-1}q^{\frac{3}{2}\ell ^{2} +\frac{1}{2}\ell}\sum_{n=0}^{\infty} \frac{ q^{(2\ell+1)n +n^{2}}(1-q^{\ell})}{(q;q)_{n} (q;q)_{n+2\ell} (1-q^{n+\ell})}
\\[8pt]
&	=1+(-1)^{\ell-1}q^{\frac{3}{2}\ell ^{2} +\frac{1}{2}\ell}\sum_{i=0}^{\ell-1} \sum_{n=0}^{\infty} \frac{ q^{(2\ell+1)n +n^{2}+i} }{(q;q)_{n} (q^{2};q)_{n+2\ell-1} (1-q^{n+\ell})} .
\end{aligned}\]
Obviously, the coefficients of $q^n$ in the sum of the last expression  are nonnegative and thus  (\ref{MK--cor-ineq}) holds.

Moreover, recall that a pair of sequences
$(\gamma_n)_{n\geq 0}$ and $(\delta_n)_{n\geq 0}$ is called
a conjugate Bailey pair relative to $a$ if
\begin{equation}\label{CBP}
\gamma_n=\sum_{r=n}^{\infty} \frac{\delta_r}{(q;q)_{r-n}(aq;q)_{r+n}}
\end{equation}
for all $n\geq 0$.  We note that the above Theorem \ref{equi-1} also can be derived by constructing the following conjugate Bailey pair.

\begin{lem}\label{ConBP} For $n\geq 0$,
\begin{align}
\gamma_n&=\frac{b^n(aq^{n+1}/b;q)_{\infty}}{(aq;q)_{\infty}}\,
\sum_{r=n}^{\infty}
\frac{1}{(q;q)_{r-n}(1-bq^r)}\Big(\frac{aq^{n+1}}{b}\Big)^r,\nonumber \\
\delta_n&=\frac{a^n q^{n(n+1)}}{1-bq^n},\label{newCBP-2}
\end{align}
form a conjugate Bailey pair.
\end{lem}

\begin{proof} 	To show that the sequences $\gamma_n$ and $\delta_n$ form a conjugate Bailey pair, it is equivalent to prove the following identity
\[
\sum_{r=n}^{\infty}
\frac{b^n (aq^{n+1}/b;q)_{\infty}}
{(q;q)_{r-n}(1-bq^r)}\Big(\frac{aq^{n+1}}{b}\Big)^r=
\sum_{r=n}^{\infty} \frac{a^r q^{r(r+1)}(aq^{r+n+1};q)_{\infty}}
{(q;q)_{r-n}(1-bq^r)}.
\]
By using the identity \cite[Equation (II.2)]{Basic}
\begin{align*}
(z;q)_{\infty}=\sum_{k=0}^{\infty}\frac{(-z)^k q^{\binom{k}{2}}}{(q;q)_k}
\end{align*}
to expand both infinite products on both sides of the above identity  and then extracting the coefficients of	$(-a)^k q^{\binom{k}{2}+(n+1)k}$, it turns to be that for all $k\geq n\geq 0$
\[
\sum_{r=n}^k \frac{(-1)^r b^{n-k} q^{\binom{r+1}{2}-rk}}
{(q;q)_{k-r}(q;q)_{r-n}(1-bq^r)}=
\sum_{r=n}^k \frac{(-1)^r q^{\binom{r+1}{2}-rn}}
{(q;q)_{k-r}(q;q)_{r-n}(1-bq^r)}.
\]
Further writing it as a single sum, shifting the summation index $r\mapsto r+n$
then replacing $(b,k)\mapsto (bq^{-n},k+n)$ and simplifying, the above identity is equivalent to
\[
\sum_{r=0}^k (-1)^r q^{\binom{r+1}{2}-rk}
\frac{1-(bq^r)^k}{1-bq^r} {k\brack r}=0.
\]
Denoting the summation on the left by $L_k(b,q)$, the coefficient of $ b^{i} $ by $ [b^{i}]L_k(b,q) $,
it follows from the formula \eqref{formulag} that if $0\leq i\leq k-1$
\[
[b^i] L_k(b,q)=
\sum_{r=0}^k (-1)^r q^{\binom{r}{2}-(k-i-1)r} {{k}\brack {r}}=0.
\]
By the $q$-binomial theorem \cite[Equation (II.4)]{Basic}
\[
(z;q)_n=\sum_{k=0}^n(-1)^k {{n}\brack{k}}q^{\binom{k}{2}}z^{k},
\]
we obtain
\[
\sum_{r=0}^k (-1)^r q^{\binom{r}{2}-(k-i-1)r}{{k}\brack{r}}
=(q^{1+i-k};q)_k=0
\]
for all $0\leq i\leq k-1$, which deduce that \eqref{newCBP-2} form a conjugate Bailey pair.
\end{proof}

We note that by substituting the conjugate Bailey pair  \eqref{newCBP-2} into \eqref{CBP} and letting $ a=b=1 $, it implies
\begin{align*}
\frac{1}{(q;q)_n}\sum_{r=n}^{\infty}
\frac{q^{(n+1)r}}{(q;q)_{r-n}(1-q^r)}
&=\sum_{r=n}^{\infty}
\frac{q^{r(r+1)}}{(q;q)_{r-n}(q;q)_{r+n}(1-q^r)} \\
&=\sum_{r=0}^{\infty}
\frac{q^{(r+n)(r+n+1)}}{(q;q)_r(q;q)_{r+2n}(1-q^{r+n})},
\end{align*}
which is equivalent to Theorem \ref{equi-1} by taking the substitution $(r,n)\mapsto (n,\ell)$.

Next, we prove the  truncated Gauss   theta series on triangular numbers as given in Theorem \ref{thm-2}.

{\noindent \it Proof of Theorem \ref{thm-2}}.
By inserting the Bailey pair (\ref{Bailey-pair-1}) into (\ref{pf-12}) and setting $ \rho =-a^{\frac{1}{2}}$, we are led to
\[ \begin{aligned}
& \sum_{m_1 \geq m_2 \geq \cdots \geq m_k \geq 0} \frac{(-a^{\frac{1}{2}})_{m_1} a^{\frac{m_1}{2}+m_2+\cdots+m_k} q^{\frac{m_1\left(m_1-1\right)}{2}+m_2^2+m_3^2+\cdots+m_k^2-m_2-m_3-\cdots-m_i}   (a^{\frac{1}{2}} q^{-\frac{1}{2}})_{m_k}}{(q)_{m_1-m_2}(q)_{m_2-m_3} \cdots(q)_{m_{k-1}-m_k} (a,q,a^{\frac{1}{2}} q^{\frac{1}{2}})_{m_{k}}   } \notag \\[8pt]
&=  \frac{( -a ^{\frac{1}{2}})_{\infty}}{(a )_{\infty}}
\Big(\sum_{n \geq 0}   (-1)^n  a^{k n} q^{ k n^2-\left(\frac{1}{2}+i\right) n} + \sum_{n \geq 1} (-1)^n  a^{k n+i-k+\frac{1}{2}} q^{kn^2+\left(\frac{1}{2}+i-2 k\right) n+k-i-\frac{1}{2}}\Big) .
\end{aligned} \]
Replacing $ a $ by $ q^{2\ell +1} $, it leads to that
\[ \begin{aligned}
& \sum_{m_1 \geq m_2 \geq \cdots \geq m_k \geq 0} \frac{(-q^{\ell +\frac{1}{2}})_{m_1} (q^{2\ell +1})^{\frac{m_1}{2}+m_2+\cdots+m_k} q^{\frac{m_1\left(m_1-1\right)}{2}+m_2^2+m_3^2+\cdots+m_k^2-m_2-m_3-\cdots-m_i} (q^{\ell})_{m_k}}{(q)_{m_1-m_2}(q)_{m_2-m_3} \cdots(q)_{m_{k-1}-m_k}  (q^{2\ell+1},q,q^{\ell+1} )_{m_{k}} } \notag \\[8pt]
&=  \frac{(-q^{\ell+\frac{1}{2}})_{\infty}}{(q^{2\ell+1} )_{\infty}}
\sum_{n \geq 0}   (-1)^n  (q^{2\ell+1})^{k n} q^{ k n^2-\left(\frac{1}{2}+i\right) n} \\[8pt]
&\quad+\frac{(-q^{\ell+\frac{1}{2}})_{\infty}}{(q^{2\ell+1} )_{\infty}}  \sum_{n \geq 1} (-1)^n  (q^{2\ell+1})^{k n+i-k+\frac{1}{2}} q^{kn^2+\left(\frac{1}{2}+i-2 k\right) n+k-i-\frac{1}{2}} \\[8pt]
&=(-1){^{\ell}}q^{-k\ell ^{2} -(k-i-\frac{1}{2}) \ell }\frac{(-q^{\frac{1}{2}+\ell})_{\infty}}{(q^{2\ell+1})_{\infty}}
\Big(\sum_{n \leq -\ell} +  \sum_{n \geq \ell+1} \Big)  (-1)^{n}q^{ k n^2+\left(\frac{1}{2}+i-k\right) n}.
\end{aligned} \]
Then substituting $ k=1 $, $ i=1 $ and replacing $ q $ by $ q^{2} $, we obtain that
\begin{align}\label{pf-2-2}
& q^{2\ell^{2} -\ell }\frac{(q^{2};q^{2})_{\infty}}{(-q;q^{2})_{\infty}}\sum_{n=0}^{\infty} \frac{ (-q;q^{2})_{n+\ell} q^{(2\ell+ 1)n +{n(n-1)} }  (1-q^{2\ell})  }{(q^{2};q^{2})_{n} (q^{2};q^{2})_{n+2\ell } (1-q^{2n+2\ell})}\notag \\[8pt]
& =(-1)^{\ell} \sum_{n \leq -\ell}(-1)^n q^{  2n^2+ n} + (-1)^{\ell} \sum_{n \geq \ell+1} (-1)^n  q^{  2n^2+ n}.
\end{align}
Therefore,
\[ \begin{aligned}
&\frac{\left(-q ; q^2\right)_{\infty}}{\left(q^2 ; q^2\right)_{\infty}} \sum_{n=-\ell +1}^{\ell}(-1)^n q^{2 n^2+n}\\[8pt]
&=1+(-1)^{\ell -1} q^{2\ell^{2} -\ell }  \sum_{n=0}^{\infty}\frac{ (-q;q^{2})_{n+\ell} q^{(2\ell+ 1)n +{n(n-1)} }  (1-q^{2\ell})  }{(q^{2};q^{2})_{n} (q^{2};q^{2})_{n+2\ell } (1-q^{2n+2\ell})},
\end{aligned} \]
which completes the proof of Theorem \ref{thm-2}. \qed

Further expanding the right hand side of the above equation, we have  \[ \begin{aligned}
&\frac{\left(-q ; q^2\right)_{\infty}}{\left(q^2 ; q^2\right)_{\infty}} \sum_{n=-\ell +1}^{\ell}(-1)^n q^{2 n^2+n}\\[8pt]
&\quad =1+(-1)^{\ell -1} q^{2\ell^{2} -\ell } \sum_{i=0}^{\ell-1} \sum_{n=0}^{\infty}\frac{ (-q;q^{2})_{n+\ell} q^{(2\ell+1)n +{n(n-1)} +2i}   }{(q^{2};q^{2})_{n} (q^{4};q^{2})_{n+2\ell-1 } (1-q^{2n+2\ell})}.
\end{aligned} \]
Obviously, the coefficients of $q^n$ for $n\geq 0$  in the sum on the right hand side are nonnegative. Thus by taking the coefficients of $q^n$ on both hand sides and combining with the generating function for $\pod(n)$ \eqref{gf-pod}, we obtain the following result.

\begin{cor}\label{thm-2-cor}
For $n, \ell \geq 1$, we have
\begin{align}
(-1)^{\ell-1} \sum_{j=-\ell+1 }^{\ell}(-1)^j{\pod}\big(n-j(2 j+1)\big)\geq 0
\end{align}
with strict inequality if $n \geq(2 \ell+1) \ell$.
\end{cor}

Then we give the proof of the truncated Gauss'  theta series on square numbers.

\noindent{\it Proof of Theorem  \ref{thm-3}.}
Taking $ i=0 $ in Lemma  \ref{bailey-lattice}, we have that
\begin{align}\label{pf-thm-3}
& \sum_{n \geq m_1 \geq m_2 \geq \ldots>m_k \geq 0} \frac{\left(\rho_1\right)_{m_1} \ldots\left(\rho_k\right)_{m_k}\left(\sigma_1\right)_{m_2} \ldots\left(\sigma_k\right)_{m _{k}}}{(q)_{n-m_1} \ldots(q)_{m _{k-1}-m _{k}}}  \notag\\[8pt]
& \times \frac{\left(a q / \rho_{1} \sigma_{1}\right)_{n-m_{1}} \ldots\left(a q / \rho_k \sigma_k\right)_{m _{k-1}-m_k} a^{m_1+\ldots+m_k} q^{m_{1}+\ldots+m_k} \beta_{m _{k}}}{\left(a q / \rho_{1}\right)_{n} \ldots\left(a q / \rho_k\right)_{m _{k-1}}\left(a q / \sigma_{1}\right)_{n} \ldots\left(a q / \sigma_k\right)_{m _{k-1}}\left(\rho_1 \sigma_1\right)^{m_1} \ldots\left(\rho_k \sigma_k\right)^{m_k}} \notag\\[8pt]
& =\frac{\alpha_0}{(q)_n(a)_n}+\sum_{t=1}^n \frac{ (1-a)}{(q)_{n-t}(a)_{n+t}}  \notag\\[8pt]
&
\times\bigg(\frac{\left(\rho_{1}\right)_t\left(\sigma_{1}\right)_t \ldots\left(\rho_k\right)_t\left(\sigma_k\right)_t(a q)^{k t} \alpha_t}{\left(a q / \rho_{1}\right)_t\left(a q / \sigma_{1}\right)_t \ldots\left(a q / \rho_k\right)_t\left(a q / \sigma_k\right)_t\left(\rho_{1} \cdots \rho_k \sigma_{1} \cdots \sigma_k\right)^t\left(1-a q^{2 t}\right)} \notag\\[8pt]
&-\frac{\left(\rho_{1}\right)_{t-1}\left(\sigma_{1}\right)_{t-1} \cdots\left(\rho_k\right)_{t-1}\left(\sigma_k\right)_{t-1}(a q)^{k(t-1)} a q^{2 t-2} \alpha_{t-1}}{\left(a q / \rho_{1}\right)_{t-1}\left(a q / \sigma_{1}\right)_{t-1} \cdots\left(a q / \rho_k\right)_{t-1}\left(a q / \sigma_k\right)_{t-1}\left(\rho_{1} \sigma_{1}\cdots \rho_k  \sigma_k\right)^{t-1}\left(1-a q^{2 t-2}\right)} \bigg).
\end{align}
Taking $n\to \infty$, $\rho_{s} \to \infty$ for $s=2,3,\ldots, k$,  $\sigma_{s}\to \infty$ for $ s=1,2,\ldots k$, and then letting $\rho_{1} \mapsto \rho$ in (\ref{pf-thm-3}), it becomes
\begin{align}\label{pf-3-1}
& \sum_{m_1 \geq m_2 \geq \cdots \geq m_k \geq 0} \frac{(\rho)_{m_1}(-\rho)^{-m_1}  a^{{m_1}+m_2+\cdots+m_k} q^{\frac{m_1^{2} +m_{1}}{2}+m_2^2+m_3^2+\cdots+m_k^2} \beta_{m_k}}{(q)_{m_1-m_2}(q)_{m_2-m_3} \cdots(q)_{m_{k-1}-m_k}} \notag \\[8pt]
& =  \frac{(aq/\rho)_{\infty}}{(a q)_{\infty}} \sum_{n \geq 0} \frac{(\rho)_{n}(-\rho)^{-n}  a^{k n} q^{\left(k-\frac{1}{2}\right) n^2+\frac{1}{2} n} \alpha_n}{ (aq/\rho)_{n}\left(1-a q^{2 n}\right)} \notag\\[8pt]
& \quad -\frac{(aq / \rho)_{\infty}}{(a q)_{\infty}} \sum_{n \geq 1} \frac{ (\rho)_{n-1}(-\rho)^{-n+1}   a^{kn-k+1} q^{\left(k-\frac{1}{2}\right) n^2+\left(\frac{7}{2}-2 k\right) n+k-3} \alpha_{n-1}}{(aq/\rho)_{n-1}\left(1-a q^{2 n-2}\right)}.
\end{align}
Inserting the Bailey pair (\ref{Bailey-pair-1}) into (\ref{pf-3-1}) and substituting $ \rho =-(aq)^{\frac{1}{2}} $ yields that
\[ \begin{aligned}
& \sum_{m_1 \geq m_2 \geq \cdots \geq m_k \geq 0} \frac{(-a^{\frac{1}{2}}q^{\frac{1}{2}})_{m_1} a^{\frac{m_1}{2}+m_2+\cdots+m_k} q^{\frac{m_1^2}{2}+m_2^2+m_3^2+\cdots+m_k^2} (1-a^{\frac{1}{2}}q^{-\frac{1}{2}})}{(q)_{m_1-m_2}(q)_{m_2-m_3} \cdots(q)_{m_{k-1}-m_k}  (q)_{m_{k}} (a)_{m_{k}} (1-a^{\frac{1}{2}}q^{-\frac{1}{2}+m_{k}}) } \notag \\[8pt]
&=  \frac{(-a^{\frac{1}{2}}q^{\frac{1}{2}})_{\infty}}{(a )_{\infty}} \sum_{n \geq 0} (-1)^{n}a^{kn} q^{kn^{2} -n}  +\frac{(-a^{\frac{1}{2}}q^{\frac{1}{2}})_{\infty}}{(a )_{\infty}} \sum_{n \geq 1} (-1)^{n}a^{kn-k+1} q^{kn^{2} +(1-2k) n+k-1  }.
\end{aligned} \]
Setting $ a=q^{2\ell+1} $ and $ k=1 $, then we have
\begin{align}\label{squ-pf}
&\sum_{n=0}^{\infty } \frac{(-q)_{n+\ell}  q^{(\ell+1)n+\frac{n(n-1)}{2}}  (1-q^{\ell}) }{(q)_{n} (q)_{n+2\ell} (1-q^{n+\ell})}\notag\\[8pt]
&=\frac{(-q;q)_{\infty}}{(q;q)_{\infty}} \Big(\sum_{n\geq 0}(-1)^{n} q^{n^{2} +2\ell n} +\sum_{n\geq 1}(-1)^{n} q^{n^{2} +2\ell n} \Big) \notag\\[8pt]
&=\frac{(-q;q) _{\infty}}{(q;q)_{\infty}}(-1)^{\ell} q^{-\ell^{2}} \Big(\sum_{n\leq -\ell }(-1)^{n} q^{n^{2} } +\sum_{n\geq \ell+1}(-1)^{n} q^{n^{2}}\Big).
\end{align}
By applying the Jacobi triple product identity \eqref{JTPI}, it implies that
\[ \begin{aligned}
&\frac{(-q ; q)_{\infty}}{(q ; q)_{\infty}} \sum_{n=-\ell+1 }^{\ell}(-1)^n q^{n^2}\\[8pt]
&=1+(-1)^{\ell-1} q^{\ell^{2}}\frac{(-q ; q)_{\infty}}{(q ; q)_{\infty}}\frac{(q ; q)_{\infty}}{(-q ; q)_{\infty}}  \sum_{n=0}^{\infty } \frac{(-q)_{n+\ell}  q^{(\ell+1)n+\frac{n(n-1)}{2}} (1-q^{\ell})  }{(q)_{n} (q)_{n+2\ell} (1-q^{n+\ell})},
\end{aligned} \]
which completes the proof of  Theorem \ref{thm-3}. \qed

Moreover, by noting that
\[ \begin{aligned}
&\frac{(-q ; q)_{\infty}}{(q ; q)_{\infty}} \sum_{n=-\ell+1 }^{\ell}(-1)^n q^{n^2}\\[8pt]
&=1+(-1)^{\ell-1} q^{\ell^{2}}\sum_{i=0}^{\ell-1}\sum_{n=0}^{\infty } \frac{(-q)_{n+\ell}  q^{(\ell+1)n+\frac{n(n-1)}{2}+i}   }{(q)_{n} (q^{2})_{n+2\ell-1} (1-q^{n+\ell})},
\end{aligned} \]
then taking the coefficients of $q^n$ on both hand sides and   combining the generating function for overpartitions \eqref{gf-over}, we obtain the following result.

\begin{cor}\label{thm-3-cor}
For $n, \ell \geq 1$, we have
\begin{align}
(-1)^{\ell-1} \sum_{j=-\ell+1}^\ell(-1)^j \overline{p}\left(n-j^2\right) \geq 0,
\end{align}
with strict inequality if $n \geq\ell^{2}$.
\end{cor}

At last in this section, we provide another truncated form for Gauss' theta series on square numbers (\ref{Gauss-2}).

\begin{thm}\label{equi-gauss-2} We have
\[ \frac{(-q ; q)_{\infty}}{(q ; q)_{\infty}} \sum_{n=-\ell+1 }^{\ell}(-1)^n q^{n^2}= 1+(-1)^{\ell-1} q^{{\ell}^{2}} \frac{(-q)_{\ell}}{(q)_{\ell-1}} \sum_{n=0}^{\infty } \frac{(-q^{n+\ell+1})_{\infty}}{( q^{n+\ell+1})_{\infty}}(q^{\ell})^{n}.\]
\end{thm}

\begin{proof} By applying Theorem   \ref{thm-3}, we see that
\begin{align}
&\quad \sum_{n=0}^{\infty } \frac{(-q)_{n+\ell}  q^{(\ell+1)n+\frac{n(n-1)}{2}+\ell^{2}}  (1-q^{\ell}) }{(q)_{n} (q)_{n+2\ell} (1-q^{n+\ell})} \notag  \\[8pt]
&=q^{\ell^{2}}\frac{(-q)_{\ell}}{(q)_{2\ell}} \sum_{n=0}^{\infty } \frac{(-q^{\ell+1})_{n}  q^{(\ell+1)n+\frac{n(n-1)}{2}+\ell^{2}}  (q^{\ell})_{n}}{(q)_{n} (q^{2\ell+1})_{n} (q^{\ell+1})_{n}}  \notag  \\[8pt]
&=q^{\ell^{2}}\frac{(-q)_{\ell}}{(q)_{2\ell}} \frac{(-q^{\ell+1})_{\infty} (-q^{\ell+1})_{\infty}}{(q^{\ell+1})_{\infty}(q^{2\ell+1})_{\infty}} \sum_{n=0}^{\infty } \frac{(-q^{\ell})_{n} (-1)_{n}}{(-q^{\ell+1})_{n} (q)_{n}} (-q^{\ell+1})^n  \notag  \\[8pt]
&\quad (\text{  by }(\ref{3phi2-2}) \text{ with }  c\to \infty,\ a=q^{\ell},\ b= -q^{\ell+1},\ d= q^{\ell+1},\ e=q^{2\ell+1})\notag \\[8pt]
&=q^{\ell^{2}}\frac{(-q)_{\ell}}{(q)_{2\ell}} \frac{(-q^{\ell+1})_{\infty} (-q^{\ell+1})_{\infty}}{(q^{\ell+1})_{\infty}(q^{2\ell+1})_{\infty}} \frac{(q^{\ell})_{\infty}}{(-q^{\ell+1})_{\infty}} \sum_{n=0}^{\infty } \frac{(q^{\ell+1})_{n}}{(-q^{\ell+1})_{n}}(q^{\ell})^{n}  \notag \\[8pt]
&\quad (\text{  by }\ (\ref{Heine-2}) \ \text{ with }\   a= -q^{\ell},\ b=-1,\ z=c=-q^{\ell+1})  \notag  \\[8pt]
&=q^{\ell^{2}}\frac{(-q)_{\ell}}{(q)_{\ell-1}} \sum_{n=0}^{\infty } \frac{(-q^{n+\ell+1})_{\infty}}{({ q^{n+\ell+1}})_{\infty}}(q^{\ell})^{n}.\notag
\end{align}
\end{proof}

Based on the above result, we note that  Theorem \ref{thm-3} also can be derived with the help of Rogers--Fine identity \eqref{R-R-id} by letting $ \beta\mapsto -\alpha $, $ \tau \mapsto \alpha/q $ and $ \alpha\mapsto -q^{\ell+1} $.

\section{Ballantine and Merca's conjectures}\label{BM-con}
In a recent work \cite{ballantine20236}, Ballantine and Merca considered 6-regular partitions  and some related truncated forms. Together with Euler's pentagonal number theorem  (\ref{Euler}), they found that
\[(-1)^k\Big(b_6(n)-\sum_{j=-k+1}^k (-1)^j p\big(n-3 j(3 j-1)\big)\Big)=\sum_{j=0}^{\lfloor n / 6\rfloor} b_6(n-6 j) M_{k}(j) ,  \]
where $ M_{k}(j) $ is the same as defined by (\ref{trun-euler}).
Furthermore, they derived a   recurrence relation for $  b_6(n) $ as follows
\begin{align}\label{Merca recurrrence}
\sum_{j=-\infty}^{\infty}(-1)^j b_6(n-j(3 j-1) / 2)= \begin{cases}(-1)^{k}, & \text { if } n=3k(3k-1), k \in \mathbb{Z}, \\ 0, & \text { otherwise. }\end{cases}
\end{align}
Inspired by (\ref{Euler}) and (\ref{Merca recurrrence}), they considered the following series
\begin{equation}\label{BMConj}
(-1)^k \left((q^6 ; q^6)_{\infty}-\frac{\left(q^6 ; q^6\right)_{\infty}}{(q ; q)_{\infty}} \sum_{n=-k+1}^k(-1)^n q^{n(3 n-1) / 2} \right)\end{equation}
and conjectured that the series has nonnegative coefficients, which is equivalent to Conjecture \ref{B-M-conj-1}.  In this section, we shall prove this conjecture in a more general case.

Recall that  $p_{5}(n)$ denotes the number of partitions of $n$ only with parts $1$, $2$, $4$ or $5$ and $p_5(n)=0$ if $n<0$. From OEIS  \cite[A029007]{OEIS}, we know that
\[ p_{5}(n)= \left\lfloor\frac{2 n^3+36 n^2+193 n+525}{480}+\frac{(-1)^n(n+1)}{32}\right\rfloor, \]
where $\lfloor a \rfloor$ represents the largest integer not greater than $a$.
Then it is easy to see that
\begin{equation}\label{boundsp5}
p^{d}_{5}(n)< p_{5}(n)\leq p^{u}_{5}(n),
\end{equation}
where
\[
\begin{aligned}
p^{d}_{5}(n)&=\frac{2 n^3+36 n^2+193 n+525}{480}-\frac{n+1}{32}-1,\\
p^{u}_{5}(n)&=\frac{2 n^3+36 n^2+193 n+525}{480}+\frac{n+1}{32}.
\end{aligned} \]
To prove  Conjecture \ref{B-M-conj-1} and the more general results, we mainly rely on the following theorem.

\begin{thm}\label{coe-p5}
For $n\geq 0$, $k \geq 1$, the coefficient  of $q^n$ in the series
\begin{equation}\label{p5series}
\frac{1}{(1-q)(1-q^2)(1-q^4)(1-q^{5})} \sum_{j=0}^{\infty}(-1)^j q^{\left(3 j^2+6 jk+j\right) / 2}\left(1-q^{2 j+2k+1}\right)
\end{equation}
is nonnegative.
\end{thm}

\begin{proof} By applying the generating function \eqref{gf-p5} for $p_5(n)$, we can write the  coefficients of $q^n$ in \eqref{p5series} as follows
\[	C(n)=\sum_{j=0}^\infty (-1)^j \Big( p_5\big(n-\frac{3 j^2+6 kj+j}{2}\big)-p_5\big(n-\frac{3 j^2+6 kj+5j+4k +2}{2} \big)\Big),
\]
in which there are finite number of  nonzero terms   since $p_5(n)=0$ if $n<0$.
Then consider the parities of $j$, it turns to be
\begin{align}		C(n)&=\sum_{j=0}^\infty \Big(   p_5(n-6 j^2-6k j-j)-p_5(n-6 j^2-6k j-5 j-2k-1)\label{p5cn}\\
&  -p_5(n-6 j^2-6k j-7 j-3k-2)+p_5(n-6 j^2-6k j-11 j-5k-5)\Big).\nonumber
\end{align}
We consider the partial sum of  $C(n)$ and denote it by
\[ \begin{aligned}
C(n,n_{0})&=  \sum_{j=0}^{n_0-1}\Big(   p_5(n-6 j^2-6k j-j)-p_5(n-6 j^2-6k j-5 j-2k-1)\\
& \  -p_5(n-6 j^2-6k j-7 j-3k-2)+p_5(n-6 j^2-6k j-11 j-5k-5)\Big).
\end{aligned} \]
Applying the boundary conditions  \eqref{boundsp5} for $p_5(n)$, it is obvious to see that
\begin{align}\notag
C(n,n_{0})&> \sum_{j=0}^{n_0-1}\Big(   p^{d}_5(n-6 j^2-6k j-j)-p^{u}_5(n-6 j^2-6k j-5 j-2k-1)\notag\\
& \quad  -p^{u}_5(n-6 j^2-6k j-7 j-3k-2)+p^{d}_5(n-6 j^2-6k j-11 j-5k-5)\Big) \notag\\
&=\sum_{j=0}^{n_0-1}\Big(-\frac{9 j^4}{2}-9 j^3(1+k)-\frac{1}{4 0} j(1+k)\big(-125+162k+36k^2-36 n\big)\notag\\[8pt]
&\quad -\frac{1}{40} j^2(55+522k+216k^2-36 n)+\frac{1}{80}\big(-30k^3+k^2(-31+12 n)\notag\\[8pt]
&\quad +2k(59+17 n)-2(7 n+n^2)-133\big)\Big)\notag\\[8pt]
&=-\frac{n^2 n_0}{40} +n \Big(\frac{3k^2 n_0}{20}+\frac{9k n_0^2}{20}-\frac{k n_0}{40}+\frac{3 n_0^3}{10}-\frac{19 n_0}{40}\Big)\notag\\
& \quad 	+\left(-\frac{9k^2}{5}+\frac{3k}{20}+\frac{61}{24}\right)n_0^3+\left(-\frac{9k^3}{20}+\frac{9k^2}{40}+\frac{61k}{16}\right) n_0^2\notag\\
&\quad +\left(\frac{3k^3}{40}+\frac{19k^2}{16}-\frac{19k}{80}-\frac{793}{240}\right) n_0-\frac{9k }{4}n_0^4-\frac{9}{10} n_0^5,\notag
\end{align}
where we denote this lower bound for $C(n,n_0)$ by $C^{d}(n, n_0)$.

By noting that the set of nonnegative integers can be decomposed into the union of the following intervals
\[
\mathbb{N}=\bigcup_{n_0\geq 0} \big[6 {n_0}^2+6k n_0+n_0, 6 ({n_0}+1)^2+6k (n_0+1)+n_0\big],
\]
thus to show that $C(n)\geq 0$ for $n\geq 0$, it is sufficient to verify that for any given $n_0\geq 0$, it holds for $n$ belonging to the interval
\[ 6 n_0^2+6k n_0+n_0\leq n \leq 6 (n_0+1)^2+6k (n_0+1)+n_0. \]
Next,  we  will further decompose it into four cases.

{\noindent \bf  Case 1:} $ 6 n_0^2+6k n_0+n_0\leq n \leq 6 n_0^2+6k n_0+5n_0 +2k$.

In this case, from \eqref{p5cn}, it is easy to see that
\[
C(n)=C(n,n_0)+p_5(n-6 n_{0}^2-6k n_{0}-n_{0}).
\]
Thus
\[ \begin{aligned}
C(n)&>	C^{d}(n, n_0)+p^{d}_5(n-6 n_{0}^2-6k n_{0}-n_{0})\\
&=\frac{n^3}{240}-n^2 \Big(\frac{3k n_0}{40}+\frac{3 n_0^2}{40}+\frac{3 n_0}{80}-\frac{3}{40}\Big)+n \Big(\frac{9}{20} k^2 n_0^2+\frac{3k^2 n_0}{20}+\frac{9k n_0^3}{10}+\frac{3k n_0^2}{5}  \\
&\quad -\frac{37k n_0}{40}+\frac{9 n_0^4}{20}+\frac{9 n_0^3}{20}-\frac{71 n_0^2}{80}-\frac{5 n_0}{8}+\frac{89}{240}\Big)-\bigg(\frac{27k^2}{10}+\frac{63k}{20}-\frac{21}{8}\bigg) n_0^4\\
&\quad +\left(-\frac{9k^3}{10}-\frac{9k^2}{4}+\frac{219k}{40}+\frac{55}{16}\right) n_0^3+\left(-\frac{9k^3}{20}+\frac{117k^2}{40}+\frac{377k}{80}-\frac{43}{20}\right) n_0^2\\	&\quad +\left(\frac{3k^3}{40}+\frac{19k^2}{16}-\frac{197k}{80}-\frac{147}{40}\right) n_0+\left(-\frac{27k}{10}-\frac{27}{20}\right) n_0^5-\frac{1}{10} 9 n_0^6+\frac{1}{16}.
\end{aligned}\]
Denote 	\[ C_{1}(n)=C^{d}(n, n_0)+p^{d}_5(n-6 n_{0}^2-6k n_{0}-n_{0}). \]	

When  $ n_{0}=0 $, it is obvious that
\[
C(n)> C_{1}(n)=\frac{n^3}{240}+\frac{3n^2}{40}+\frac{89n}{240}+\frac{1}{16}> 0,
\]
for all $ n\geq 0. $

When $ k=1 $ and $ n_{0}=1 $, we see that $ 13\leq n\leq 19 $. Then
\[
C(n)>C_{1}(n)= \frac{n^3}{240}-\frac{9 n^2}{80}+\frac{14 n}{15}-\frac{35}{16}>0.
\]

When $ k\geq 1 $ and $ n_{0}\geq 2 $ or $ k\geq2 $ and $ n_{0}\geq 1 $, the extremum of the first derivation of  $C_{1}(n) $ is greater than 0, i.e.,
\[ \begin{aligned}
C_{1}^{\prime}(n)&\geq C_{1}^{\prime}(6 n_{0}^2+6k n_{0}+3n_{0}-6)\\	&=\left(\frac{3k}{20}-\frac{1}{10}\right)n_0^2+\left(\frac{3k^2}{20}-\frac{k}{40}-\frac{7}{40}\right) n_0 -\frac{19}{240}>0.
\end{aligned} \]
Thus  $ C_{1}(n) $ is  monotonically increasing,  and thereby
\[ \begin{aligned}
C(n)&> C_{1}(6 n_{0}^2+6k n_{0}+n_{0} )\\
&=\frac{9k n_0^4}{20}+		\left(\frac{9k^2}{10}+\frac{3k}{20}-\frac{1}{3}\right) n_0^3+\left(\frac{9k^3}{20}+\frac{9k^2}{40}+\frac{15k}{16}-\frac{19}{40}\right) n_0^2\\	&\quad+\left(\frac{3k^3}{40}+\frac{19k^2}{16}-\frac{19k}{80}-\frac{793}{240}\right) n_0+ \frac{1}{16}> 0.
\end{aligned}
\]

{\noindent\bf  Case 2:}  $ 6 n_0^2+6k n_0+5n_0+2k+1 \leq n \leq 6 n_0^2+6k n_0+7n_0+3k+1 $.

In this case, from \eqref{p5cn}, it can be seen that
\[ C(n)=C(n,n_{0})+p_5(n-6 n_{0}^2-6k n_{0}-n_{0})-p_5(n-6 n_{0}^2-6k n_{0}-5n_{0}-2k-1). \]
Then,
\[ \begin{aligned}
C(n)&> C^{d}(n,n_{0})+p^{d}_5(n-6 n_{0}^2-6k n_{0}-n_{0})-p^{u}_5(n-6 n_{0}^2-6k n_{0}-5n_{0}-2k-1)\\
&=n^2 \left(\frac{k}{40}+\frac{n_0}{40}+\frac{1}{80}\right)+n \left(-\frac{3k^2 n_0}{20}-\frac{k^2}{20}-\frac{9k n_0^2}{20}-\frac{17k n_0}{40}+\frac{k}{4}-\frac{3 n_0^3}{10} \right.\\
&\quad \left.-\frac{9 n_0^2}{20}+\frac{3}{40}\right)+\frac{k^3}{30}+\left(\frac{9k^2}{5}+\frac{87k}{20}+\frac{5}{24}\right) n_0^3-\frac{k^2}{4}+\frac{9 n_0^5}{10}+\left(\frac{9k}{4}+\frac{9}{4}\right) n_0^4\\
&\quad +\left(\frac{3k^3}{8}-\frac{k^2}{16}-\frac{31k}{16}-\frac{527}{240}\right) n_0+\left(\frac{9k^3}{20}+\frac{99k^2}{40}+\frac{7k}{80}-\frac{31}{16}\right) n_0^2+\frac{71k}{120}-\frac{7}{10}.			
\end{aligned} \]
Denote
\[ C_{2}(n)=C^{d}(n,n_{0})+p^{d}_5(n-6 n_{0}^2-6k n_{0}-n_{0})-p^{u}_5(n-6 n_{0}^2-6k n_{0}-5n_{0}-2k-1). \]
When $ n_{0}\geq0 $, we can compare the axis of symmetry of the quadratic function $C_{2}(n)  $ with the lower bound of $ n $ in this case. It can be seen that
\[ \begin{aligned}
&-[n]C_{2}(n)-2[n^{2}]C_{2}(n)\times(6 n_0^2+6k n_0+5n_0+2k+1)\\
&=-\frac{3 k^2 n_0}{20}-\frac{k^2}{20}-\frac{9 k n_0^2}{20}-\frac{7 k n_0}{40}-\frac{7 k}{20}-\frac{1}{10} 3 n_0^3-\frac{n_0^2}{5}-\frac{9 n_0}{40}-\frac{1}{10}< 0,
\end{aligned}\]
where $ [x^{k}]f(x) $ denotes the coefficient of $ x^{k} $ in polynomial $ f(x)$. Combining it with
\[ [n^{2}]C_{2}(n)=\left(\frac{k}{40}+\frac{n_0}{40}+\frac{1}{80}\right)>0 , \]
we obtain
\[\frac{-[n]C_{2}(n)}{2[n^{2}]C_{2}(n)}-(6 n_0^2+6k n_0+5n_0+2k+1)<0. \]
Thus $ C_{2}(n) $ is monotonically increasing in Case 2 by the properties of quadratic functions, and
\[
\begin{aligned}
&C(n)> C_{2}(6 n_0^2+6k n_0+5n_0+2k+1)\\
&=\frac{k^3}{30}+\left(\frac{9 k^2}{10}+\frac{3 k}{4}-\frac{2}{3}\right) n_0^3+\frac{7 k^2}{20}+\left(\frac{9 k^3}{20}+\frac{9 k^2}{8}+\frac{71 k}{80}-\frac{49}{40}\right) n_0^2\\
&\quad +\left(\frac{3 k^3}{8}+\frac{111 k^2}{80}+\frac{7 k}{80}-\frac{401}{240}\right) n_0+\frac{9 k n_0^4}{20}+\frac{16 k}{15}-\frac{49}{80}>0.
\end{aligned} \]

{\noindent\bf  Case 3:} $  6 n_0^2+6k n_0+7n_0+3k+2 \leq n \leq 6 n_0^2+6k n_0+11n_0+5k+4$.

In this case, from \eqref{p5cn}, we can see that
\[\begin{aligned}
C(n)&=C(n,n_{0})+p_5(n-6 n_{0}^2-6k n_{0}-n_{0})-p_5(n-6 n_{0}^2-6k n_{0}-5n_{0}-2k-1)\\
&\quad-p^{u}_5(n-6 n_{0}^2-6k n_{0}-7n_{0}-3k-2).
\end{aligned}\]
Then,
\[ \begin{aligned}
C(n)&>C^{d}(n,n_{0})+p^{d}_5(n-6 n_{0}^2-6k n_{0}-n_{0})-p^{u}_5(n-6 n_{0}^2-6k n_{0}-5n_{0}-2k-1)\\
&\quad-p^{u}_5(n-6 n_{0}^2-6k n_{0}-7n_{0}-3k-2)\\
&=-\frac{n^3}{240}+n^2 \left(\frac{3 k n_0}{40}+\frac{k}{16}+\frac{3 n_0^2}{40}+\frac{9 n_0}{80}-\frac{3}{80}\right)-n \left(\frac{9}{20} k^2 n_0^2+\frac{3 k^2 n_0}{5}+\frac{13 k^2}{80}\right.\\
&\quad \left.+\frac{9 k n_0^3}{10}+\frac{39 k n_0^2}{20}+\frac{7 k n_0}{20}-\frac{11 k}{20}+\frac{1}{20} 9 n_0^4+\frac{27 n_0^3}{20}+\frac{37 n_0^2}{80}-\frac{7 n_0}{10}+\frac{13}{120}\right)\\
&\quad -\frac{7 k^2}{10}+\frac{9}{10} k^3 n_0^3+\frac{9}{5} k^3 n_0^2+\frac{21 k^3 n_0}{20}+\frac{7 k^3}{48}+\frac{27}{10} k^2 n_0^4+\frac{153}{20} k^2 n_0^3+\frac{9}{2} k^2 n_0^2\\
&\quad -\frac{43 k^2 n_0}{40}+\frac{27 k n_0^5}{10}+\frac{99 k n_0^4}{10}+\frac{303 k n_0^3}{40}-\frac{163 k n_0^2}{40}-\frac{47 k n_0}{16}+\frac{137 k}{120}+\frac{9 n_0^6}{10}\\
&\quad +\frac{81 n_0^5}{20}+\frac{33 n_0^4}{8}-\frac{41 n_0^3}{16}-\frac{263 n_0^2}{80}-\frac{73 n_0}{80}-\frac{49}{40}.
\end{aligned} \]
Denote
\[\begin{aligned}
C_{3}(n)&=C^{d}(n,n_{0})+p^{d}_5(n-6 n_{0}^2-6k n_{0}-n_{0})-p^{u}_5(n-6 n_{0}^2-6k n_{0}-5n_{0}-2k-1)\\
&\quad-p^{u}_5(n-6 n_{0}^2-6k n_{0}-7n_{0}-3k-2).
\end{aligned} \]

When $ n_{0}=0 $ and $ k=1 $,  then we have $5 \leq n \leq 9 $ and
\[ C_{3}(n)= -\frac{n^3}{240}+\frac{n^2}{40}+\frac{67 n}{240}-\frac{51}{80}>0.\]

When  $ n_{0}\geq 1 $ or $ k\geq 2 $, by considering the first derivation $ C_{3}^{\prime}(n) $ of  $C_{3}(n) $, we find that $ C_{3}^{\prime}(n)  $ is greater than 0 since
\[ \begin{aligned}
&C_{3}^{\prime}(  6 n_0^2+6k n_0+7n_0+3k+2)\\
&\qquad=\left(\frac{3k^2}{20}+\frac{9k}{40}+\frac{11}{40}\right) n_0+\frac{k^2}{10}+\left(\frac{3k}{20}+\frac{1}{20}\right) n_0^2+\frac{17k}{40}-\frac{37}{120}>0,\\
&C_{3}^{\prime}(6 n_0^2+6k n_0+11n_0+5k+4)\\
&\qquad=\left(\frac{3 k^2}{20}+\frac{13 k}{40}-\frac{11}{80}\right) n_0+\frac{3 k^2}{20}+\left(\frac{3 k}{20}+\frac{1}{20}\right) n_0^2+\frac{7 k}{40}-\frac{73}{120}>0,
\end{aligned} \]
It deduces $ C_{3}(n) $ is monotonically increasing in Case 3, and
\[ \begin{aligned}
&C(n)> C_{3}(6 n_0^2+6k n_0+7n_0+2)\\
&=\frac{13k^3}{120}+\left(\frac{9k^2}{10}+\frac{21k}{20}-\frac{2}{3}\right) n_0^3+\frac{13k^2}{16}+\left(\frac{9k^3}{20}+\frac{63k^2}{40}+\frac{107k}{80}-\frac{31}{40}\right) n_0^2\\
&\quad +\left(\frac{21k^3}{40}+\frac{147k^2}{80}+\frac{99k}{80}-\frac{293}{240}\right) n_0+\frac{9k n_0^4}{20}+\frac{47k}{30}-\frac{13}{8}>0.
\end{aligned} \]

{\noindent\bf  Case 4:} $ 6 n_0^2+6k n_0+11n_0+5k+5\leq n  \leq 6 (n_0+1)^2+6k (n_0+1)+n_0$.

In this case, $C(n)=C(n,n_0+1)$. We can directly compute $ C^{d}(n,n_{0}+1) $ for $ n_{0}\geq 0 $ and $ k\geq 1 $. By noting that
\[  \begin{aligned}
&	C^{d}( 6 n_0^2+6kn_0 +11n_0+5k+5,n_{0}+1)\\
&=\left(\frac{9k^2}{10}-\frac{3k}{20}-\frac{1}{3}\right) (n_0+1)^3+\left(\frac{9k^3}{20}-\frac{9k^2}{40}+\frac{15k}{16}+\frac{19}{40}\right) (n_0+1)^2\\
&\quad +\left(-\frac{3k^3}{40}+\frac{19k^2}{16}+\frac{19k}{80}-\frac{793}{240}\right) (n_0+1)+\frac{9k (n_0+1)^4}{20}>0,\\
&C^{d}( 6 (n_0+1)^2+6k (n_0+1)+n_0,n_{0}+1)\\
&=\left(\frac{9k^2}{10}+\frac{3k}{20}-\frac{1}{3}\right) (n_0+1)^3+\left(\frac{9k^3}{20}+\frac{9k^2}{40}+\frac{15k}{16}-\frac{19}{40}\right) (n_0+1)^2\\
&\quad +\left(\frac{3k^3}{40}+\frac{19k^2}{16}-\frac{19k}{80}-\frac{793}{240}\right) (n_0+1)+\frac{9k (n_0+1)^4}{20}>0,
\end{aligned}  \]
we deduce that
\[ \begin{aligned}
C(n)>&\text{min}\{C^{d}( 6 n_0^2+6kn_0 +11n_0+5k+5,n_{0}+1),\\
&\quad \quad C^{d}( 6 (n_0+1)^2+6k (n_0+1)+n_0,n_{0}+1)  \} >0.
\end{aligned} \]
Above all, we complete the  proof.
\end{proof}

Now, we are ready to show the more general case compared to \eqref{BMConj}.

\begin{thm}\label{conj-regular}
For $n\geq 0$, $ \ell=3 $ or $\ell
\geq 6 $, the coefficient of $ q^{n} $ in	
\begin{align}\label{general-regular}
(-1)^{k}\Big(\left(q^{\ell} ; q^{\ell}\right)_{\infty}-\frac{\left(q^{\ell} ; q^{\ell}\right)_{\infty}}{(q ; q)_{\infty}} \sum_{j=-k+1}^k(-1)^j q^{j(3 j-1) / 2}\Big)
\end{align}  is nonnegative.
\end{thm}

{\noindent \bf  Remark.} From Yao's result \cite[Theorem  1.3]{Yao6regular} and following the similar procedures as given below, we can see that the above result still holds for  $\ell\geq 4$. Therefore, we can conclude that when $ \ell \geq 3 $,  \eqref{general-regular} has nonnegative coefficients.

\begin{proof} We can see that
\[\begin{aligned}
&\quad \ (-1)^{k}\Big(\left(q^{\ell} ; q^{\ell}\right)_{\infty}-\frac{\left(q^{\ell} ; q^{\ell}\right)_{\infty}}{(q ; q)_{\infty}} \sum_{j=-k+1}^k(-1)^j q^{j(3 j-1) / 2}\Big) \\
& = (-1)^{k}\frac{\left(q^{\ell} ; q^{\ell}\right)_{\infty}}{(q;q)_{\infty}} \Big(\sum_{j=-\infty}^{-k}+\sum_{j=k+1}^{\infty}\Big)
(-1)^j q^{j(3j-1) / 2} \\
& =q^{\frac{3 {k}^2+{k}}{2}}  \frac{\left(q^{\ell} ; q^{\ell}\right)_{\infty}}{(q;q)_{\infty}} \sum_{j=0}^{\infty}(-1)^j q^{\left(3 j^2+6 j {k}+j\right) / 2}\left(1-q^{2 j+2 {k}+1}\right) \\
& = q^{\frac{3 {k}^2+{k}}{2}}   \frac{(q^{\ell}; q^{\ell})_{\infty}}{(1-q^{3}) \left(q^{6}; q\right)_{\infty}}  \frac{\sum_{j=0}^{\infty}(-1)^j q^{\left(3 j^2+6 j {k}+j\right) / 2}\left(1-q^{2 j+2 {k}+1}\right) }{(1-q)(1-q^2)(1-q^4)(1-q^{5})}
\\[8pt]
& = q^{\frac{3 {k}^2+{k}}{2}}   \frac{(q^{\ell}; q^{\ell})_{\infty}}{(1-q^{3}) \left(q^{6}; q\right)_{\infty}}  \frac{\sum_{j=0}^{\infty}(-1)^j q^{\left(3 j^2+6 j {k}+j\right) / 2}\left(1-q^{2 j+2 {k}+1}\right) }{(1-q)(1-q^2)(1-q^4)(1-q^{5})}.
\end{aligned}\]
Obviously, when $ \ell=3 $ or $\ell
\geq 6 $, the coefficients in
\[   \frac{(q^{\ell}; q^{\ell})_{\infty}}{(1-q^{3}) \left(q^{6}; q\right)_{\infty}} \]
are nonnegative. Combine it with Theorem \ref{coe-p5}  the proof is complete.
\end{proof}
Clearly, when $\ell=6$, Ballantine and Merca's Conjecture \ref{B-M-conj-1} for  $ 6 $-regular partitions is confirmed.

Finally, we show that when $R=3S$, the stronger conjecture, i.e. Conjecture \ref{stronger-ja}, for Jacobi triple product series  holds.

{\noindent \it Proof of Theorem \ref{R=3-stronger}.}
Taking   $ \ell=3 $ in the proof of Theorem \ref{conj-regular}, we have that
\begin{align}
&\quad \ q^{\frac{3 {k}^2+{k}}{2}}  \frac{\left(q^{3} ; q^{3}\right)_{\infty}}{(q;q)_{\infty}} \sum_{j=0}^{\infty}(-1)^j q^{\left(3 j^2+6 j {k}+j\right) / 2}\left(1-q^{2 j+2 {k}+1}\right)\notag \\
&=   q^{\frac{3 {k}^2+{k}}{2}}   \frac{ 1  }{\left(q^{7},q^{8} ; q^{3}\right)_{\infty}}  \frac{\sum_{j=0}^{\infty}(-1)^j q^{\left(3 j^2+6 j {k}+j\right) / 2}\left(1-q^{2 j+2 {k}+1}\right) }{(1-q)(1-q^2)(1-q^4)(1-q^{5})} \notag \\
&=q^{\frac{3 {k}^2+{k}}{2}}  \frac{1}{(q,q^{2};q^{3})_{\infty}} \sum_{j=0}^{\infty}(-1)^j q^{\left(3 j^2+6 j {k}+j\right) / 2}\left(1-q^{2 j+2 {k}+1}\right) \notag \\
&= q^{\frac{3 {k}^2+{k}}{2}}  \frac{1}{(q,q^{2};q^{3})_{\infty}} \bigg(\sum_{j=0}^{\infty}(-1)^j q^{(3 j^2+6 j {k}+j) / 2}- \sum_{j=0}^{\infty}(-1)^j q^{(3 j^2+6 j {k}+5j+4k+2) / 2}\bigg)\notag \\[8pt]
&= (-1)^{k} \frac{\sum_{j=k}^{\infty}(-1)^j q^{ 3j(j+1) / 2}\left(q^{-j}-q^{j+1}\right)}{\left(q, q^{2}; q^{3}\right)_{\infty}}.\notag
\end{align}
From Theorem \ref{conj-regular}, it implies that the above series has nonnegative coefficients. Thereby, the proof can be completed by  directly replacing $ q $ with $ q^{S} $ in the above expression. \qed

\section{Concluding remarks}\label{conclusion}

Firstly, we remark that if we only consider the truncated forms for  the three classical  theta series,   it is sufficient to use the following two lemmas due to Agarwal, Andrews and Bressoud \cite{agarwal1987bailey} for the Bailey chain and Bailey lattice.

\begin{lem}
If $ (\alpha_{n},\beta_n) $ is a Bailey pair relative to $ a $, then so is $ (\alpha_{n}^{\prime},\beta_{n}^{\prime}) $, where
$$
\begin{aligned}
\alpha_n^{\prime}&=\frac{(\rho)_n(\sigma)_n(a q / \rho \sigma)^n \alpha_n}{(a q / \rho)_n(a b / \sigma)_n},\\[8pt]
\beta_n^{\prime}&=\sum_{k=0}^n \frac{(a q / \rho \sigma)_{n-k}(\rho)_k(\sigma)_k(a q / \rho \sigma)^k \beta_k}{(q)_{n-k}(a q / \rho)_n(a q / \sigma)_n},
\end{aligned}
$$
and
\begin{align}\label{lem-1}
\sum_{k=0}^n \frac{(a q / \rho \sigma)_{n-k}(\rho)_k(\sigma)_k(a q / \rho \sigma)^k \beta_k}{(q)_{n-k}(a q / \rho)_n(a q / \sigma)_n}=\sum_{r=0}^{n} \frac{\alpha_{r}^{\prime}}{(q)_{n-r}(a q)_{n+r}}.
\end{align}
\end{lem}

\begin{lem}
If $ (\alpha_{n},\beta_n) $ is a Bailey pair relative to $ a $, then $ (\alpha_{n}^{\prime},\beta_{n}^{\prime}) $ is a Bailey pair relative to $ a/q$, where
$$
\begin{aligned}
\alpha_n^{\prime}&=(1-a)\left(\frac{a}{\rho \sigma}\right)^n \frac{(\rho)_n(\sigma)_n}{(a / \rho)_n(a / \sigma)_n}\left\{\frac{\alpha_n}{1-a q^{2 n}}-\frac{a q^{2 n-2} \alpha_{n-1}}{1-a q^{2 n-2}}\right\}, \\[8pt]
\beta_n^{\prime}&=\sum_{k=0}^n \frac{(a  / \rho \sigma)_{n-k}(\rho)_k(\sigma)_k(a  / \rho \sigma)^k \beta_k}{(q)_{n-k}(a / \rho)_n(a  / \sigma)_n},
\end{aligned}
$$
and
\begin{align}\label{lem-2}
\sum_{k=0}^n \frac{(a q / \rho \sigma)_{n-k}(\rho)_k(\sigma)_k(a  / \rho \sigma)^k \beta_k}{(q)_{n-k}(a  / \rho)_n(a / \sigma)_n}=\sum_{r=0}^{n} \frac{\alpha_{r}^{\prime}}{(q)_{n-r}(a )_{n+r}}.
\end{align}
\end{lem}
For example, by inserting the Bailey pair (\ref{Bailey-pair-1}) into (\ref{lem-2}), then setting $ n\to\infty $, $ \rho\mapsto-a^{\frac{1}{2}} $, $ \sigma\to\infty  $, and replacing $ q  $ by $ q^{2} $, it also leads to  (\ref{pf-2-2}).

Secondly, in Section~4, we show that our truncated form for   Euler's pentagonal number theorem \eqref{thm-1-eq} is equivalent to Andrews and Merca's result \eqref{trun-euler}. The coefficients of $q^n$ in our  expression \eqref{thm-1-eq} also can be interpreted in terms of certain partitions of $n$ whose number is definitely equal to $M_k(n)$. The remaining problem is how to construct a bijection between these two kinds of partitions.


\section*{Acknowledgments}
The authors would like to thank Ole Warnaar for valuable comments and suggestions  on the earlier version of this manuscript, especially for the initial derivation of the conjugate Bailey pair as given in Lemma 4.2.  We also thank the anonymous referee for the specific suggestions that helped us to improve the exposition.
This work was supported by the National Natural Science Foundation of China (Grant No. 12071235) and the Fundamental Research Funds for the Central Universities.


\begin{thebibliography}{99}

\bibitem{adamovic2009n} D. Adamovi\'{c}, A. Milas, The $  N= 1 $ triplet vertex operator superalgebras, Comm. Math. Phys. 288 (2009) 225--270.

\bibitem{adamovic2008n} D. Adamovi\'{c}, A. Milas,  The $ N= 1 $ triplet vertex operator superalgebras: twisted sector, SIGMA. 4 (2008) 087.

\bibitem{AFSHARIJOO2023108946} P. Afsharijoo, J. Dousse, F. Jouhet, H. Mourtada, New companions to the Andrews--Gordon identities motivated by commutative algebra, Adv.  Math. 417 (2023) 108946.

\bibitem{agarwal1987bailey} A. Agarwal, G. Andrews, D. Bressoud, The Bailey lattice, J. Indian Math. Soc. (N.S.). (1987) 57--73.

\bibitem{andrews1975problems} G.E. Andrews, Problems and prospects for basic hypergeometric functions, in: Theory and Application of Special Functions, Elsevier, 1975: pp. 191--224.

\bibitem{Partitions} G.E. Andrews, The Theory of Partitions, Cambridge University Press,  Cambridge, UK, 1998.

\bibitem{andrews2012truncated} G.E. Andrews, M. Merca, The truncated pentagonal number theorem, J. Combin. Theory, Ser. A 119 (2012) 1639--1643.

\bibitem{andrews2018truncated}G.E. Andrews, M. Merca, Truncated theta series and a problem of Guo and Zeng, J. Combin. Theory, Ser. A 154 (2018) 610--619.

\bibitem{andrews1999} G.E. Andrews, A. Schilling, S.O. Warnaar, An $A_2$ Bailey lemma and Rogers-Ramanujan-type identities, J. Amer. Math. Soc. 12 (1999) 677--702.

\bibitem{Bailey1949} W.N. Bailey, Identities of the Rogers--Ramanujan type, Proc. London Math. Soc. (2) 50 (1948) 1--10.

\bibitem{ballantine2024truncated} C. Ballantine, B. Feigon, Truncated theta series related to the Jacobi triple product identity, Discrete Math. 348 (2025) 114319.

\bibitem{ballantine20236} C. Ballantine, M. Merca, 6-regular partitions: new combinatorial properties, congruences, and linear inequalities, Rev. Real Acad. Cienc. Exactas Fis. Nat. Ser. A-Mat. 117 (2023).

\bibitem{Berkovich1998} A. Berkovich, B.M. McCoy, A. Schilling, Rogers-Schur-Ramanujan type identities for the $M(p, p^{\prime})$ minimal models of conformal field theory, Commun. Math. Phys. 191 (1998)
325--395.

\bibitem{bressoud1980analytic} D.M. Bressoud, Analytic and combinatorial generalizations of the Rogers--Ramanujan identities, Mem. Amer. Math. Soc. 24 (1980) 1-54.

\bibitem{corteel2004overpartitions} S. Corteel, J. Lovejoy, Overpartitions, Trans. Amer. Math. Soc. 356 (2004), 1623--1635.

\bibitem{Feigin1993} B. Feigin, E. Frenkel, Coinvariants of nilpotent subalgebras of the  Virasoro algebra and partition identities,
I.M. Gelfand Seminar, 139--148, Adv. Soviet Math., 16, Part 1, Amer. Math. Soc., Providence, RI, 1993.

\bibitem{Basic} G. Gasper, M. Rahman, Basic Hypergeometric Series, 2nd Ed., Cambridge University Press, Cambridge, 2004.

\bibitem{gordon1961combinatorial} B. Gordon, A combinatorial generalization of the Rogers--Ramanujan identities, Amer. J. Math. 83 (1961) 393--399.

\bibitem{GuoZeng13}	V.J.W. Guo, J. Zeng, Two truncated identities of Gauss, J. Combin. Theory, Ser. A 120 (2013) 700--707.


\bibitem{he2016bilateral} T.Y. He, K.Q. Ji, W.J. Zang, Bilateral truncated Jacobi's identity, European J. Combin. 51 (2016) 255--267.

\bibitem{hirschhorn2010arithmetic} M.D. Hirschhorn, J.A. Sellers, Arithmetic properties of partitions with odd parts distinct, Ramanujan J. 22 (2010) 273--284.

\bibitem{RepJames} G.D. James, A. Kerber, The Representation Theory of the Symmetric Group, In: Encyclopedia of Mathematics and its Applications, Vol. 16, Addison-Wesley, Reading, MA, 1981.

\bibitem{jennings2020q} C. Jennings-Shaffer, A. Milas, On $q$-series identities for false theta series, Adv. Math. 375 (2020) 107411.

\bibitem{kolitsch2015interpreting} L.W. Kolitsch, M. Burnette, Interpreting the truncated pentagonal number theorem using partition pairs, Electron. J. Combin. 22(2) (2015).

\bibitem{LEPOWSKY197815} J. Lepowsky, S. Milne, Lie algebraic approaches to classical partition identities, Adv. Math. 29 (1978) 15--59.

\bibitem{Liuexpan} Z.-G. Liu, On the $ q $-partial differential equations and $ q $-series, Ramanujan Math. Soc. Lect. Notes Ser. 20 (2013) 213--250.

\bibitem{mao2015proofs} R. Mao, Proofs of two conjectures on truncated series, J. Combin. Theory, Ser. A 130 (2015) 15--25.

\bibitem{merca2022two} M. Merca, On two truncated quintuple series theorems, Experiment. Math. 31 (2022) 606--610.

\bibitem{Rogers1917} L.J. Rogers, On two theorems of combinatory analysis and some allied identities, Proc. London Math. Soc. 16 (1917) 315--336.

\bibitem{shanks1951short} D. Shanks, A short proof of an identity of Euler, Proc. Amer. Math. Soc. 2 (1951) 747--749.

\bibitem{OEIS} N.J.A. Sloane, On-line Encyclopedia of Integer Sequences, Available:
http://oeis.org/, 1964.

\bibitem{wang2019truncated} C. Wang, A.J. Yee, Truncated Jacobi triple product series, J. Combin. Theory, Ser. A 166 (2019) 382--392.

\bibitem{wang2020truncated}C. Wang, A.J. Yee, Truncated Hecke--Rogers type series, Adv. Math. 365 (2020) 107051.

\bibitem{Warnaar1996} S.O. Warnaar, Fermionic solution of the Andrews-Baxter-Forrester model. II. Proof of Melzer's polynomial identities, J. Stat. Phys. 84 (1996) 49--83.

\bibitem{xia2022new} E.X. Xia, A.J. Yee, X. Zhao, New truncated theorems for three classical theta function identities, European J. Combin. 101 (2022) 103470.

\bibitem{Yao6regular} O.X.M. Yao, Proof of a conjecture of Ballantine and Merca on truncated sums of 6-regular partitions, J. Combin. Theory, Ser. A 206 (2024) 105903.

\bibitem{yee2015truncated} A.J. Yee, A truncated Jacobi triple product theorem, J. Combin. Theory, Ser. A 130 (2015) 1--14.

\end{thebibliography}

\end{document}